\documentclass[final,onefignum,onetabnum]{siamonline190516}

\usepackage{braket,amsfonts}

\usepackage{array}

\usepackage[caption=false]{subfig}

\newsiamremark{example}{Example}
\newsiamremark{experiment}{Experiment}
\newsiamremark{hypothesis}{Hypothesis}
\crefname{hypothesis}{Hypothesis}{Hypotheses}

\usepackage{algorithmic}

\usepackage{graphicx,epstopdf}

\Crefname{ALC@unique}{Line}{Lines}

\usepackage{amsopn}

\usepackage{xspace}
\usepackage{bold-extra}
\usepackage[most]{tcolorbox}

\colorlet{texcscolor}{blue!50!black}
\colorlet{texemcolor}{red!70!black}
\colorlet{texpreamble}{red!70!black}
\colorlet{codebackground}{black!25!white!25}

\lstdefinestyle{siamlatex}{%
  style=tcblatex,
  texcsstyle=*\color{texcscolor},
  texcsstyle=[2]\color{texemcolor},
  keywordstyle=[2]\color{texemcolor},
  moretexcs={cref,Cref,maketitle,mathcal,text,headers,email,url},
}

\tcbset{%
  colframe=black!75!white!75,
  coltitle=white,
  colback=codebackground, 
  colbacklower=white, 
  fonttitle=\bfseries,
  arc=0pt,outer arc=0pt,
  top=1pt,bottom=1pt,left=1mm,right=1mm,middle=1mm,boxsep=1mm,
  leftrule=0.3mm,rightrule=0.3mm,toprule=0.3mm,bottomrule=0.3mm,
  listing options={style=siamlatex}
}

\DeclareTotalTCBox{\code}{ v O{} }
{ 
  fontupper=\ttfamily\color{black},
  nobeforeafter,
  tcbox raise base,
  colback=codebackground,colframe=white,
  top=0pt,bottom=0pt,left=0mm,right=0mm,
  leftrule=0pt,rightrule=0pt,toprule=0mm,bottomrule=0mm,
  boxsep=0.5mm,
  #2}{#1}

\patchcmd\newpage{\vfil}{}{}{}
\usepackage{braket,amsfonts}
\usepackage{subfig}
\usepackage{hyperref}
\usepackage{mathtools}
\usepackage{xcolor}
\usepackage{bm}
\usepackage{cleveref}
\newcommand{\norm}[1]{\lVert#1\rVert}
\newcommand{\R}{\mathbb R}
\newcommand{\N}{\mathbb N}
\newcommand{\be}{\mathbf e}
\newcommand{\bq}{\mathbf q}
\newcommand{\bu}{\mathbf u}
\newcommand{\bc}{\mathbf c}
\newcommand{\br}{\mathbf r}

\newcommand{\bx}{\mathbf x}
\newcommand{\by}{\mathbf y}
\newcommand{\bz}{\mathbf z}
\newcommand{\bss}{\mathbf s}
\newcommand{\bp}{\mathbf p}

\newcommand{\calo}{\mathcal O}
\newcommand{\calq}{\mathcal Q}
\newcommand{\calz}{\mathcal Z}
\newcommand{\wh}{\widehat}
\newcommand{\wt}{\widetilde}
\DeclareMathOperator*{\argmax}{\arg\!\max}
\usepackage{color}
\definecolor{green}{rgb}{0.0, 0.5, 0.0}
\definecolor{orange}{rgb}{1.0,0.4,0}
\definecolor{aqua}{rgb}{0.0, 0.0, 1.0}
\definecolor{darkmagenta}{rgb}{0.55, 0.0, 0.55}
\colorlet{yellow}{green!30!orange!70!}

\begin{tcbverbatimwrite}{tmp_\jobname_header.tex}
\title{A Generalized CUR decomposition for matrix pairs\thanks{Version \today.\funding{This work has received funding from the European Union's Horizon 2020 research and innovation programme under the Marie Sklodowska-Curie grant agreement No 812912.}}}
\author{Perfect~Y.~Gidisu\thanks{Department of Mathematics and Computer Science, TU Eindhoven, The Netherlands, (\email{p.gidisu@tue.nl}, \email{m.e.hochstenbach@tue.nl}).} \and Michiel~E.~Hochstenbach\footnotemark[2]}

\headers{A Generalized CUR decomposition for matrix pairs}{Perfect Y. Gidisu and Michiel E. Hochstenbach}
\end{tcbverbatimwrite}
\input{tmp_\jobname_header.tex}

\ifpdf
\hypersetup{ pdftitle={A Generalized CUR decomposition for matrix pairs} }
\fi

\begin{document}
\maketitle
\begin{tcbverbatimwrite}{tmp_\jobname_abstract.tex}
\begin{abstract}
We propose a generalized CUR (GCUR) decomposition for matrix pairs $(A,B)$. Given matrices $A$ and $B$ with the same number of columns, such a decomposition provides low-rank approximations of both matrices simultaneously, in terms of some of their rows and columns. We obtain the indices for selecting the subset of rows and columns of the original matrices using the discrete empirical interpolation method (DEIM) on the generalized singular vectors.
When $B$ is square and nonsingular, there are close connections between the GCUR of $(A,B)$ and the DEIM-induced CUR of $AB^{-1}$.
When $B$ is the identity, the GCUR decomposition of $A$ coincides with the DEIM-induced CUR decomposition of $A$.
We also show similar connection between the GCUR of $(A,B)$ and the CUR of $AB^+$ for a nonsquare but full-rank matrix $B$, where $B^+$ denotes the Moore--Penrose pseudoinverse of $B$. While a CUR decomposition acts on one data set, a GCUR factorization jointly decomposes two data sets. The algorithm may be suitable for applications where one is interested in extracting the most discriminative features from one data set relative to another data set. In numerical experiments, we demonstrate the advantages of the new method over the standard CUR approximation; for recovering data perturbed with colored noise and subgroup discovery. 
\end{abstract}


\end{tcbverbatimwrite}
\input{tmp_\jobname_abstract.tex}

\section{Introduction}\label{sec:intro}
With the proliferation of big data matrices, dimension reduction has become an important tool in many data analysis applications. There are several methods of dimension reduction for a given problem; however, the approximated data often consists of derived features that are either no longer interpretable or difficult to interpret in the original context.  For example, while the singular value decomposition (SVD) provides an optimal approximation and compression of data, it may be difficult for domain experts to directly draw conclusions or interpret the singular vectors. In some applications, it is necessary to find a dimension reduction method that preserves the original properties (such as sparsity, nonnegativity, being integer-valued) of the data and ensures interpretability. In an attempt to solve this difficulty, one possibility for a low-rank representation of a given data matrix is to use a subset of the original columns and rows of the matrix itself: a CUR decomposition; see, e.g., Mahoney and Drineas \cite{Drineas}. The selected subsets of rows and columns capture the most relevant information of the original matrix.

A CUR decomposition of rank $k$ of a (square or rectangular) $m \times n$ matrix $A$ is of the form
\begin{equation}
\label{cur}
A \ \approx \ CMR \ := \ A P \, \cdot \, M \, \cdot \, S^TA.
\end{equation}
 Here, $P$ is an $n \times k$ (where $k < \min(m,n)$) {\em index selection matrix} with some columns of the identity matrix $I_n$ that selects certain columns of $A$.
Similarly, $S$ is an $m \times k$ matrix with columns of $I_m$
that selects certain rows of $A$; so $C$ is $m \times k$ and $R$ is $k \times n$. We construct the $k \times k$ matrix $M$ in such a way that the decomposition has some desirable approximation properties that we will discuss in \cref{sec:GCUR}. (In line with \cite{Str19}, we will use the letter $M$ rather than $U$.) It is also possible to select a different number of columns and rows; here, $M$ will not be of dimension $k\times k$. It is worth noting that given $k$, this decomposition is not unique; there are several ways to obtain this form of approximation to $A$ with different techniques of choosing the representative columns and rows. Many algorithms for this decomposition using the truncated SVD (TSVD) as a basis have been proposed \cite{Boutsidis, Zhang, Mahoney, Drineas}. In \cite{Sorensen}, Sorensen and Embree present a CUR decomposition inspired by a discrete empirical interpolation method (DEIM) on a TSVD. Let the rank-$k$ TSVD be
\begin{equation}
\label{svd}
A \approx W_k \Psi_k Z_k^T,
\end{equation}
where the columns of $W_k$ and $Z_k$ are orthonormal,
while $\Psi_k$ is diagonal with nonnegative elements. Throughout the paper we assume a unique truncated singular value decomposition, i.e., the $k$th singular value is not equal to the $(k+1)$st singular value.

There is extensive work on CUR-type decompositions in both numerical linear algebra and theoretical computer science. In this paper, we develop a generalized CUR decomposition (GCUR) for matrix pair, for any two matrices $A$ and $B$ with the same number of columns: $A$ is $m\times n$, $B$ is $d\times n$ and of full rank. The intuition behind this generalized CUR decomposition is that we can view it as a {\em CUR decomposition of $A$ relative to $B$}. As we will see in \cref{pp3}, when $B$ is square and nonsingular, the GCUR decomposition has a close connection with the CUR of $AB^{-1}$. The GCUR is also applicable to nonsquare matrices $B$; see the examples in \cref{sec:EXP}. We show in \cref{pp3} that if $B$ is nonsquare but of a full rank, we still have a close connection between the CUR decomposition of $AB^+$ (where $B^+$ denotes the pseudoinverse of $B$) and the GCUR decomposition. Another intuition for this GCUR decomposition comes from a footnote remark by Mahoney and Drineas \cite[p.~700]{Mahoney}: ``for data sets in which a low-dimensional subspace obtained by the SVD failed to capture category separation, CUR decomposition performed correspondingly poorly''. This is evident in \cref{exp:3}.

Inspired by the work of Sorensen and Embree \cite{Sorensen}, we present a generalized CUR decomposition using the discrete empirical interpolation method.
The DEIM algorithm for interpolation indices, presented in \cite{Chaturantabut}, is a discrete variant of empirical interpolation proposed in \cite{Barrault} as a method for model order reduction for nonlinear dynamical systems. In \cite{Sorensen}, the authors used DEIM as an index selection technique for constructing the $C$ and $R$ factors of a CUR decomposition. The DEIM algorithm independently selects the column and row indices based on the right and left singular vectors of a data matrix $A$, respectively. Our new GCUR method uses the matrices obtained from the GSVD instead. Besides using DEIM on the GSVD for index selection, we can also use other CUR-type index selection strategies for the GCUR (see also \cref{sec:Con}). The proposed method can be used in situations where a low-rank matrix is perturbed with noise, where the covariance of the noise is not a multiple of the identity matrix. It may also be appropriate for applications where one is interested in extracting the most discriminative information from a data set of interest relative to another data set. We will see examples of these in \cref{sec:EXP}. 

\begin{example}
The following simple example shows that using the matrices obtained from the GSVD instead of the SVD can lead to more accurate results when approximating data with colored noise. Unlike white noise, colored noise is correlated. In discrete time, the noise samples of colored noise need not be independent. In terms of the Fourier transform, some frequencies are more present than others.
As in \cite[p.~55]{hansen} and \cite{park1994}, we use the term ``colored noise'' for noise of which the covariance matrix is not a multiple of the identity.

We consider a full-rank matrix $A_E$ representing low-rank data, and want to try to recover an original low-rank matrix perturbed by colored noise. Our test matrix $A_E$ is a rank-2 matrix $A$ of size $3\times3$ perturbed by additive colored noise $E$ with a given desired covariance structure. We take  

\[A= {\footnotesize \begin{bmatrix*}[r]
      1  &    0  &  1\\
     0   &  2    & 2\\
     1   & 1 &  2
 \end{bmatrix*}}, 
 \quad  E^T\!E= {\footnotesize \begin{bmatrix*}[l]
      1.0 & 0.8 & 0.3\\
    0.8 &  1.0 & 0.8\\
    0.3 & 0.8  & 1.0
 \end{bmatrix*}}.\]
We generate the colored noise as an additive white Gaussian noise multiplied by the Cholesky factor $(R)$ of the desired covariance matrix. The matrix $A_E$ is, as a result, a sum of a rank-2 matrix and a correlated Gaussian noise matrix. We compute the SVD of both $A$ and $A_E$. The $k$ dominant left singular vectors of $A$ are denoted by $W_k$ while those of $A_E$ are $\widetilde{W}_k$. We also compute the GSVD of $(A_E, R)$ and denote the $k$ dominant left generalized singular vectors by $U_k$. Since we are interested in recovering $A$, we examine the angle between the leading $k$-dimensional exact left singular subspace \text{Range}$(W_k)$ and its approximations \text{Range}$(\widetilde{W}_k)$ and \text{Range}$(U_k)$. We generate 1000 different test cases and take the average of the subspace angles.

\Cref{tab:0} shows the results for $k=2$ and three different noise levels. We observe that the approximations obtained using the GSVD in terms of subspace angles are more accurate than those from the SVD; about 40\% gain in accuracy. This illustrates the potential advantage of using generalized singular vectors in the presence of colored noise.
\begin{table}[htb!]
\centering
{\caption{The average angle between the leading two-dimensional exact singular subspace \text{Range}$(W_2)$ (which is the range of $A$) and its approximations \text{Range}$(\widetilde{W}_2)$ and \text{Range}$(U_2)$, for different values of the noise level $\varepsilon$. The subspaces \text{Range}$(U_2)$ and \text{Range}$(\widetilde{W}_2)$ are from the SVD of $A_E$ and the GSVD of $(A_E, R)$, respectively.}\label{tab:0}
{\footnotesize
\begin{tabular}{clc} \hline \rule{0pt}{3mm}%
$\varepsilon$ & Method & Subspace angle    \\ \hline \rule{0pt}{3.5mm}%
$5\cdot 10^{-2}$ & SVD  &$1.7\cdot 10^{-2}$   \\
& GSVD & $1.2\cdot 10^{-2}$ \\[0.5mm] \hline \rule{0pt}{3.5mm}%
$5\cdot 10^{-3}$ & SVD &$1.7\cdot 10^{-3}$   \\
& GSVD & $1.1\cdot 10^{-3}$ \\[0.5mm] \hline \rule{0pt}{3.5mm}%
$5\cdot 10^{-4}$ & SVD & $1.7\cdot 10^{-4}$  \\
& GSVD & $1.1\cdot 10^{-4}$ \\[0.5mm]\hline %
\end{tabular}}}
\end{table}
\end{example}

Inspired by this example, we expect that the GCUR compared to the CUR may produce better approximation results in the presence of non-white noise, as it is based on the GSVD instead of the SVD. We show in \cref{sec:EXP} that the GSVD and the GCUR may provide equally good approximation results even when we use an inexact Cholesky factor. 

Throughout the paper, we denote the 2-norm by $\norm{\cdot}$ and the infinity-norm by $\norm{\cdot}_\infty$. We use MATLAB notation to index vectors and matrices; thus, $A(:,\bp)$ denotes the $k$ columns of $A$ whose corresponding indices are in vector $\bp \in \N^k$.

\textbf{Outline.}
We give a brief introduction to the generalized singular value decomposition in \cref{sec:GSVD}. We also discuss the truncated GSVD and its approximation error bounds. We summarize the DEIM technique we use for index selection in \cref{sec:DEIM}. \Cref{sec:GCUR} introduces the new generalized CUR decomposition with an analysis of its error bounds. In \cref{alg:GCUR-DEIM}, we present a DEIM-type GCUR decomposition algorithm. Results of numerical experiments are presented in \cref{sec:EXP}, followed by conclusions in \cref{sec:Con}. 

\section{ Generalized singular value decomposition}\label{sec:GSVD}
The GSVD appears throughout this paper since it is a key building block of the proposed algorithm. This section gives a brief overview of this decomposition. The original proof of the existence of the GSVD has first been introduced by Van Loan in \cite{Van}. Paige and Saunders \cite{Paige} later presented a more general formulation without any restrictions on the dimensions except for both matrices to have the same number of columns. Other formulations and contributions to the GSVD have been proposed in \cite{stewart1982,sun,van1985}. For our applications in this paper, let $A\in \R^{m\times n}$ and $B\in \R^{d\times n}$ with both $m\ge n$ and $d\ge n$. Following the formulation of the GSVD proposed by Van Loan \cite{Van}: there exist matrices $U \in{ \mathbb R ^{m \times m}}$, $V \in{ \mathbb R^{d \times d}}$ with orthonormal columns and a nonsingular $X \in {\mathbb R ^{n \times n}}$ such that
\begin{equation}
\begin{aligned}
\label{gsvd}
U^T\!AX &= \Gamma = \text{diag}(\gamma_1,\dots,\gamma_n),  \qquad &\gamma_i\in [0,1],\\
V^T\!BX &=\Sigma = \text{diag}(\sigma_1,\dots,\sigma_n), \qquad &\sigma_i\in [0,1],
\end{aligned}
\end{equation}
where $\gamma_i^{2}+\sigma_i^{2} = 1 $. Although traditionally the ratios $\gamma_i/\sigma_i$ are in a nondecreasing order, for our purpose we will instead maintain a nonincreasing order. The matrices $U$ and $V$ contain the left generalized singular vectors of $A$ and $B$, respectively; and similarly, $X$ contains the right generalized singular vectors and is identical for both decompositions. While the SVD provides two sets of linearly independent basis vectors, the GSVD of $(A, B)$ gives three new sets of linearly independent basis vectors (the columns of $U, V,$ and $X$) so that the two matrices $A$ and $B$ are diagonal when transformed to these new bases.
We note that only the reduced GSVD is needed, so that we can assume that $U \in \R^{m \times n}$, $V \in \R^{d \times n}$, and $\Gamma$ and $\Sigma$ are $n \times n$.

Our analysis is based on the following formulation of the GSVD presented in \cite{van1985}. Let $Y := X^{-T}$ in the GSVD of \eqref{gsvd}, then $A = U \Gamma Y^T$ and $B = V \Sigma Y^T$. Let us characterize matrix $Y$. (In fact, Matlab's {\tt gsvd} routine renders $Y$ instead of $X$.) Since
\begin{equation}\label{eq:mgsvd}
  A=U \, \Gamma \, Y^T, \qquad B=V \, \Sigma \, Y^T,
\end{equation}
this implies that we have the following congruence transformations
\[A^T\!A=Y(\Gamma^T\Gamma)Y^T, \qquad B^T\!B=Y(\Sigma^T\Sigma)Y^T.\]
From the above, it follows that $A^T\!A$ has the same inertia as $\Gamma^T\Gamma$ and the same holds for $B^T\!B$ and $\Sigma^T\Sigma$ (here this mainly gives information on the number of zero eigenvalues).
We also see that, provided $A$ and $B$ are of full-rank, these similarity transformations hold:
\begin{equation}
\label{eqn:matrix_Y}
\begin{aligned}
 (B^T\!B)(A^T\!A)^{-1} &=Y(\Sigma^T\Sigma)(\Gamma^T\Gamma)^{-1}Y^{-1} = Y\,\text{diag}(\sigma_i^2/\gamma_i^2)\, Y^{-1},\\
 (A^T\!A)(B^T\!B)^{-1} &=Y(\Gamma^T\Gamma)(\Sigma^T\Sigma)^{-1}Y^{-1}=Y\,\text{diag}(\gamma_i^2/\sigma_i^2)\, Y^{-1}.
\end{aligned}
\end{equation}
The columns of the matrix $Y$ are therefore the eigenvectors for both $(A^T\!A)(B^T\!B)^{-1}$ and its inverse $(B^T\!B)(A^T\!A)^{-1}$. The GSVD avoids the explicit formation of the cross-product matrices $A^T\!A$ and $B^T\!B$ (see also \cref{exp:3}).

\textbf{Truncated GSVD.}
In some practical applications it could be of interest to approximate both matrices $(A, B)$ by other matrices $(A_k, B_k)$, said truncated, of a specific rank $k$. To define the truncated GSVD (TGSVD) let us partition the following matrices
\begin{equation}\label{eq:pgsvd}
U = [U_k \ \, \widehat{U}], \ V = [V_k \ \, \widehat{V}], \ Y = [Y_k \ \, \widehat{Y}], \ \Gamma = \text{diag} (\Gamma_k, \widehat{\Gamma}), \ \Sigma = \text{diag} (\Sigma_k, \widehat{\Sigma}).
\end{equation}
For use in \cref{sec:GCUR}, we define TGSVD for $(A,B)$ as (cf.~\cite[(2.34)]{hansen})
\begin{equation}
    \label{eq:tgsvd}
    A_k := U_k \Gamma_k Y_k^T, \qquad  B_k := V_k \Sigma_k Y_k^T,
\end{equation}
where $k < n$.
then it follows that $A - A_k = \widehat{U} \, \widehat{\Gamma} \, \widehat{Y}^T.$
The following proposition is useful for understanding the error bounds for the GCUR. In line with \cite[p.~495]{pchansen},
let $\psi_i(A)$ and $\psi_i(Y)$ be the singular values of matrix $A$ and $Y$, respectively (cf.~also \eqref{svd}). The first and second statements of the following proposition are from \cite{pchansen}; while the third statement may not be present in the literature yet, it is straightforward.
\begin{proposition}\label{pp1}
Let $A = U\, \Gamma \, X^{-1} = U \, \Gamma \, Y^T$ as in \eqref{gsvd}, with $Y = X^{-T}$, then for $i=1,\dots,n$ (see, e.g., \cite[pp.~495--496]{pchansen})
\[\gamma_i\cdot\psi_{\min}(Y)\leq \psi_i(A)=\psi_i(U \, \Gamma \, Y^T) \leq \psi_i(\Gamma) ~ \norm{Y}=\gamma_i\cdot\norm{Y}\]
so
\[ \frac{\psi_i(A)}{\|Y\|} \le
\gamma_i= \psi_i(\Gamma) =\psi_i(U^T\!A Y^{-T}) \leq \psi_i(A)~\norm{Y^{-1}}.\]
Moreover,
\[ {\gamma_{k+1}}\cdot\psi_{\min}(\widehat{Y}) \leq \norm {A - A_k} \leq  {\gamma_{k+1}}\cdot \norm{\widehat{Y}}.  \]
\end{proposition}
\begin{proof}
This follows from \eqref{eq:mgsvd} and the well-known property that, for the product of two matrices we have $\psi_i(A) \, \psi_{\min}{(B)} \le \psi_i(AB) \leq \psi_i(A) \, \norm{B}$ (see, e.g., \cite[p.~89]{Householder}).
\end{proof}

The results above are relevant tools for the analysis and understanding of generalized CUR and its error bounds which we will introduce in \cref{sec:GCUR}. 
\section{Discrete empirical interpolation method}\label{sec:DEIM} We now summarize the tool from existing literature \cite{Sorensen, Chaturantabut} that we use to select columns and/or rows from matrices. Besides the GSVD, the DEIM algorithm plays an important role in the proposed method. The DEIM procedure works on the columns of a specified basis vectors sequentially. The basis vectors must be linearly independent. Assuming we have a full-rank basis matrix $U \in \R^{m\times k}$ with $k\le m$, to select $k$ rows from $U$, the DEIM procedure constructs an index vector $\bss\in \N^k$ such that it has non-repeating values in $\{1,\dots,m\}$. Defining the selection matrix $S$ as an $m\times k$ identity matrix indexed by $\bss$, i.e., $S=I(:,\bss)$ and $\bx(\bss) = S^T\bx$ (cf.~\cite{Sorensen}), we have an {\em interpolatory projector} defined through the DEIM procedure as
\[\mathbb{S}=U(S^TU)^{-1}S^T.\]
We can show that $S^TU$ is nonsingular (see \cite[Lemma~3.2]{Sorensen}). The term ``interpolatory projector" stems from the fact that for any $\bx\in \R^m$ we have
\[(\mathbb{S}\bx)(\bss)=S^T\mathbb{S}\bx=S^TU(S^TU)^{-1}S^T\bx=S^T\bx=\bx(\bss),\]
implying the projected vector $\mathbb{S}\bx$ matches $\bx$ in the $\bss$ entries \cite{Sorensen}. 

To select the indices contained in $\bss$, the columns of $U$ are considered successively. The first interpolation index corresponds to the index of the entry with the largest magnitude in the first basis vector. The rest of the interpolation indices are selected by removing the direction of the interpolatory projection in the previous basis vectors from the subsequent one and finding the index of the entry with the largest magnitude in the residual vector. The index selection using DEIM is limited by the rank of the basis matrix, i.e., the number of indices selected can be no more than the number of vectors available.

To form $\bss$, let $\bu_j$ denote the $j$th column of $U$ and $U_j$ be the matrix of the first $j$ columns of $U$. Similarly, let $\bss_j$ contain the first $j$ entries of $\bss$, and let $S_j=I(:,\bss_j)$. More precisely, we define $s_1$ such that
$|\bu_1(s_1)|=\norm{\bu_1}_\infty$ and the $j$th interpolatory projector $\mathbb{S}_j$ as 
\[\mathbb{S}_j=U_j(S_j^TU_j)^{-1}S^T_j.\]
To select $s_j$, remove the $\bu_{j-1}$ component from $\bu_j$ by projecting $\bu_j$ onto indices \{$s_1$, \dots, $s_{j-1}$\}, thus 
\[\br_j=\bu_j-\mathbb{S}_{j-1}\bu_j,\]
then take the index of the entry with the largest magnitude in the residual, i.e., $s_j$ such that  
\[|\br_j(s_j)|=\norm{\br_j}_\infty.\]
As noted in \cite{Chaturantabut}, in case of a tie e.g., $|(\br_j)_i| = |(\br_j)_l|$ for $i \ne l$, the smaller index is picked. As in DEIM-induced CUR decomposition \cite[p.~A1458]{Sorensen}, this process will never produce duplicate indices. In a nutshell, we find the indices via a non-orthogonal Gram--Schmidt-like process (oblique projections) on the $\bu$-vectors. Since the input vectors are linearly independent, the residual vector $\br$ is guaranteed to be nonzero. This DEIM algorithm forces the selection matrix $S$ to find $k$ linearly independent rows of $U$ such that the local growth of $\norm{(S^TU)^{-1}}$ is kept modest via a greedy search \cite[p.~2748]{Chaturantabut} as implemented in \cref{alg:CUR-DEIM}.

\begin{tcbverbatimwrite}{tmp_\jobname_alg.tex}
\begin{algorithm}
\caption{DEIM index selection \cite{Sorensen}}
\label{alg:CUR-DEIM}
 {\bf Input:} $U \in \R^{m \times k}$, with $k\le m$ (linearly independent columns)\\
 {\bf Output:} Indices $\bss \in \N^k$ with distinct entries in $\{1,\dots,m\}$ 
\begin{algorithmic}[1]
	\STATE{$\bu = U(:,1)$}
	\STATE{ $s_1 = \argmax_{1\le i\le m}~|(\bu)_i|$}
	\FOR{ $j = 2, \dots, k$ }
	\STATE{$\bu = U(:,j)$ }
	\STATE{$\bc=U(s,1:j-1)^{-1}\bu(s)$}
	\STATE{$\br=\bu-U(:,1:j-1)\,\bc$}
	\STATE{$s_j$ = $\argmax_{1\le i\le m}~|(\br)_i|$}
	\STATE $\bss=[\bss \ \ s_j]$
	\ENDFOR
\end{algorithmic}
\end{algorithm}
\end{tcbverbatimwrite}

\input{tmp_\jobname_alg.tex}

Although the DEIM index selection procedure is basis-dependent, if the interpolation indices are determined, the DEIM interpolatory projector is independent of the choice of basis spanning the space \text{Range}$(U)$.
\begin{proposition}\label{pp2}(\cite[Def.~3.1, (3.6)]{Chaturantabut} ). 
Let $Q$ be an orthonormal basis of \text{Range}$(U)$ where $Q_i = [\bq_1,\dots,\bq_i]$ for $1 \leq i \leq k$, then
\[U(S^T U)^{-1}S^T = Q(S^T Q)^{-1}S^T.\]
\end{proposition}
This proposition allows us to take advantage of the special properties of an orthonormal matrix in cases where our input basis matrix is not (see \cref{pp4}).
\section{Generalized CUR decomposition and its approximation properties}\label{sec:GCUR}
In this section we describe the proposed generalized CUR decomposition and provide a theoretical analysis of its error bounds.
\subsection{Generalized CUR decomposition}
We now introduce a new generalized CUR decomposition of matrix pairs $(A, B)$, where $A$ is $m \times n$ $(m\ge n)$ and $B$ is $d \times n$ $(d\ge n)$, and $B$ is of full rank. This GCUR is inspired by the truncated generalized singular value decomposition for matrix pairs, as reviewed in \cref{sec:GSVD}. We now define a generalized CUR decomposition (cf.~\eqref{cur}).
\begin{definition} \label{Dfn4}
Let $A$ be $m \times n$ and $B$ be $d \times n$ and of full rank, with $m\ge n$ and $d\ge n$.
A generalized CUR decomposition of $(A, B)$ of rank $k$ is a matrix approximation of $A$
and $B$ expressed as
\begin{equation}
  \label{eq:gcur}
  \begin{aligned}
 A_k := C_A\, M_A\, R_A = AP \, M_A \, S_A^TA ~ ,  \\
 B_k := C_B\, M_B\, R_B = BP \, M_B \, S_B^TB. 
\end{aligned}
\end{equation}
Here $S_A \in \R^{m \times k}$, $S_B \in \R^{d \times k}$, and $P \in \R^{n \times k}$ are index selection matrices $(k < n)$. 
\end{definition}
It is key that {\em the same} columns of $A$ and $B$ are selected;
this gives a coupling between the decomposition of $A$ and $B$.

The matrices $C_A, C_B$ and $R_A, R_B$ are subsets of the columns and rows, respectively, of the original matrices. In the rest of the paper, we will mainly focus on the matrix $A$; we can perform a similar analysis for the matrix $B$ (see also the comments at the end of this section). As in \cref{sec:DEIM}, we again have the vectors $\bss_A, \bp$ as the indices of the selected rows and columns such that $C_A = AP$ and $R_A = S_A^TA$, where $S_A=I(:,\bss_A)$ and $P=I(:,\bp)$. The choice of $\bp$ and $\bss_A$ is based on the transformation matrices from the rank-$k$ truncated GSVD.

Given $P$ and $S_A$, the middle matrix $M_A$ can be constructed in different ways to satisfy certain desirable approximation properties. In \cite{Sorensen}, the authors show how setting $M=A(\bss,\bp)^{-1}$ leads to a CUR decomposition corresponding to the $\bp$ columns and $\bss$ rows of $A$. Instead, following these authors \cite{Sorensen} and others \cite{Mahoney,Stewart}, we choose to construct the middle matrix $M_A$ as $(C_A^T C_A)^{-1}C_A^T A R_A^T(R_A R_A^T )^{-1}$. This option, as shown by Stewart \cite{Stewart}, minimizes $\norm{A-CMR}$ for a given $\bp$ and $\bss$. Computing the middle matrix as such yields a decomposition that can be viewed as first projecting the columns of $A$ onto $\text{Range}(C)$ and then projecting the result onto the row space of $R$, both steps being optimal for the 2-norm error. 

The following proposition establishes a connection between the DEIM-GCUR of $(A, B)$ and the DEIM-CUR of $AB^{-1}$ and $AB^+$ for a square and nonsingular $B$ and a nonsquare but full-rank $B$, respectively. 
\begin{proposition}\label{pp3} (i) If $B$ is a square and nonsingular matrix, then the selected row and column indices from the CUR decomposition of $AB^{-1}$ are the same as index vectors $\bss_A$ and $\bss_B$ obtained from the GCUR decomposition of $(A, B)$, respectively.

(ii) Moreover, in the special case where $B=I$, the GCUR decomposition of $A$ coincides with the CUR decomposition of $A$, in that the factors $C$ and $R$ of $A$ are the same for both methods: the first line of \eqref{eq:gcur} is equal to \eqref{cur}.

(iii) In addition, if $B$ is nonsquare but of a full rank, we have a connection as in (i) between the indices from the CUR decomposition of $AB^+$ and the index vectors $s_A$ and $s_B$ obtained from the GCUR decomposition of $(A, B)$.
\end{proposition}
\begin{proof} (i) We start with the GSVD \eqref{eq:mgsvd}. If $B$ is square and nonsingular, then the SVD of $G=AB^{-1}$ can be expressed in terms of the GSVD of $(A,B)$, and is equal to $G = U (\Gamma\Sigma^{-1}) V^T$ \cite{matrixcomp}.
Therefore, the row index selection matrix from the SVD of $G$ is equal to $S_A$ from the GSVD of $(A,B)$; and similarly the column index selection matrix obtained from the SVD of $G$ is equal to $S_B$, since they are determined using $U$ and $V$, respectively. 

(ii) If $B=I$, then from the second line of \eqref{eq:mgsvd} we have that $Y=V\Sigma^{-1}$. This implies that the index of the largest entries in the columns of $Y$ are the same as that of $V$.
In this special case of $B=I$, we have $G = A$, so then the left and right singular vectors of $A$ are contained in the $U$ and $V$ matrices from the GSVD of $(A,I)$, respectively. Hence the selection matrix $P$ in \eqref{eq:gcur} obtained by performing DEIM on $Y$ is the same as the selection matrix $P$ in \eqref{cur} obtained by applying DEIM to the right singular vectors of $A$.

(iii) If $B$ is nonsquare but of full rank $n$, then we still have a similar connection between the GSVD of $(A,B)$ and the SVD of $AB^+$ because of the following. Since the factors in the reduced GSVD $B = V\Sigma Y^T$ are of full rank, we have $B^+ = Y^{-T} \Sigma^{-1} V^T$. This means that $AB^+ = U \Gamma \Sigma^{-1} V^T$, so the index vectors $\bss_A$ and $\bss_B$ from GCUR of $(A,B)$ are equivalent to the selected column and row indices from CUR of $AB^+$, respectively.
\end{proof}

It is worth noting that \cref{pp3} holds for DEIM-based CUR and GCUR algorithms. For alternative ways of constructing CUR and GCUR decompositions (see \cref{sec:Con}) these properties may not hold. 



Although we can obtain indices of a CUR decomposition of $AB^{-1}$ using the GCUR of $(A, B)$, the converse does not hold. We emphasize that we need the GSVD for the GCUR decomposition and cannot use the SVD of $AB^{-1}$ or $AB^+$ instead, since the GCUR decomposition requires the $Y$ matrix from \eqref{eq:tgsvd} to find the column indices. While we used the generalized singular vectors here, in principle one could use other vectors, e.g., an approximation to the generalized singular vectors.

To build the decomposition, it is relevant to know the dominant rows and columns of $A$ and $B$ in their rank-$k$ approximations. Given that $A_k$ and $B_k$ are rank-$k$ approximations of $A$ and $B$, respectively, how should the columns and rows be selected? \Cref{alg:GCUR-DEIM} is a summary of the procedure. (The backslash operator used in \cref{alg:GCUR-DEIM} is a Matlab type notation for solving linear systems and least-squares problems.)
We note that we can parallelize the work in \cref{line3,line4,line5,line6,line7,line8} since it consists of three independent runs of DEIM. Also, if we are only interested in approximating the matrix $A$ from the pair $(A, B)$, we can omit \cref{line5,line8} as well as the second part of \cref{line9}; thus saving computational cost. 

In some applications, one might be interested in a {\em generalized interpolative decomposition}, of which the column and row versions are of the form 
\begin{equation}\label{eq:ID}
A \approx C_A\wt M_A, \ B \approx C_B\wt M_B
\qquad \text{or} \qquad
A\approx \wh M_A R_A, \ B\approx \wh M_B R_B.
\end{equation}
Here $\wt M_A=C_A^+\!A$ is $k \times n$ and $\wh M_A=AR^+_A$ is $m \times k$; similar remarks hold for $\wt M_B$ and $\wh M_B$. As noted in \cite{Sorensen}, since the DEIM index selection algorithm identifies the row and column indices independently, this form of decomposition is relatively straightforward.

In terms of computational complexity, the dense GSVD method requires $\calo((m+d)n^2)$ work and the three runs of DEIM together require $\calo((m+n+d)k^2)$ work, so the overall complexity of the algorithm is dominated by the construction of the GSVD. (This might suggest iterative GSVD approaches; see \cref{sec:Con}.) 

\begin{tcbverbatimwrite}{tmp_\jobname_alg2a.tex}
\begin{algorithm}
\caption{DEIM type GCUR decomposition}
\label{alg:GCUR-DEIM}
 {\bf Input:} $A \in \R^{m \times n}$, $B \in \R^{d \times n}$ (where $m \ge n$ and $d \ge n$), desired rank $k$ \\
{\bf Output:} A rank-$k$ generalized CUR decomposition \\
$A_k = A(:,\bp) \, \cdot \, M_A \, \cdot \, A(\bss_A,:)$, \quad
$B_k = B(:,\bp) \, \cdot \, M_B \, \cdot \, B(\bss_B,:)$
\begin{algorithmic}[1]
\setcounter{ALC@unique}{0}
	\STATE{$[U, V, Y] = {\sf gsvd}(A,B)$ \hfill (according to nonincreasing generalized singular values)}
	\FOR{ $j = 1, \dots, k$ }
	\STATE\label{line3}{ $\bp(j) = \argmax_{1\le i\le n}~|(Y(\,:,j))_i|$ \quad \hfill (Iteratively pick indices)}
	\STATE\label{line4}{$\bss_A(j) = \argmax_{1\le i\le m}~|(U(\,:,j))_i|$}
	\STATE\label{line5}{$\bss_B(j) = \argmax_{1\le i\le d}~|(V(\,:,j))_i|$}\\
	 \quad \hfill (Update new columns) \\
	\STATE\label{line6}{$Y(\,:,~j+1) = Y(\,:,~j+1)-Y(:,~1:j)\cdot (Y(\bp,1:j)\ \backslash  \ Y(\bp,~j+1)$)}
	\STATE\label{line7}{$U(\,:,~j+1) = U(\,:,~j+1)-U(:,~1:j)\cdot (U(\bss_A,~1:j)\ \backslash  \ U(\bss_A,~j+1)$)}
	\STATE\label{line8}{$V(\,:,~j+1) = V(\,:,~j+1)-V(:,~1:j)\cdot (V(\bss_B,~1:j)\ \backslash \ V(\bss_B,~j+1)$)} 
	\ENDFOR
	\STATE\label{line9}{ $M_A = A(\,:,\bp) \ \backslash \ (A \ / \ A(\bss_A,:\,))$, \quad $M_B = B(\,:,\bp) \ \backslash \ (B \ / \ B(\bss_B,:\,))$}
\end{algorithmic}
\end{algorithm}
\end{tcbverbatimwrite}

\input{tmp_\jobname_alg2a.tex}

The pseudocode in \cref{alg:GCUR-DEIM} assumes the matrices from the GSVD (i.e., $U$, $V$, and $Y$) corresponds to a nonincreasing order of the generalized singular values. 

In generalizing the DEIM-inspired CUR decomposition, we also look for a generalization of the related theoretical results. While the results presented in \cite{Sorensen} express the error bounds in terms of the optimal rank-$k$ approximation, for our generalized CUR factorization, the most relevant quantity is the rank-$k$ GSVD approximation. In the following subsection, we present theoretical results for bounding the GCUR approximation error. 

\subsection{Error Bounds in terms of the SVD approximation} The error bounds for any rank-$k$ matrix approximation are usually expressed in terms of the rank-$k$ SVD approximation error. We will show a result of this type in the following proposition and also discuss its limitations. 
We introduce the following notation: let $A = W \Psi Z^T = W_k \Psi_k Z_k^T + W_{\perp} \Psi_{\perp} Z_{\perp}^T$ be the SVD of $A$ (see \eqref{svd}), where $Z_k$ contains the largest $k$ right singular vectors.
Let $Q_k$ be an $n \times k$ matrix with orthonormal columns. It turns out in both \cite{Sorensen} and this section that $\norm{A(I-Q_kQ_k^T)}$ is a central quantity in the analysis. In the DEIM-induced CUR decomposition work \cite{Sorensen}, we take the right singular vectors to be $Q_k$, but here we study this quantity for general $Q_k$. In our context, we are particularly interested in $Q_k$ as the orthogonal basis of the matrix $Y_k$ in \eqref{eq:tgsvd}.
Denote $\calq_k = \text{span}(Q_k)$ and $\calz_k = \text{span}(Z_k)$. Recall that the $\psi_i(A)$ are the singular values of $A$.

\begin{proposition}
Let $Q_k$ be an $n \times k$ matrix with orthonormal columns, and let $Z_k$ contain the largest $k$ right singular vectors of $A$. Then
\[
\psi_{k+1}^2(A) \le \|A(I-Q_kQ_k^T)\|^2 \le
\psi_{k+1}^2(A) + \|A\|^2 \cdot \sin^2(\calz_k, \calq_k).
\]
More precisely, we have
\[
\|A(I-Q_kQ_k^T)\|^2 \le \psi_{k+1}^2(A) + 
\sum_{j=1}^k \psi_j(A)^2 \cdot \sin^2(\bz_j, \calq_k).
\]
\end{proposition}
\begin{proof}
The lower bound follows from the SVD; the optimal $\calq_k$ is $\calz_k$.
We can derive the upper bounds from
\begin{align*}
\|A(I-Q_kQ_k^T)\|^2 & = \|W_k \Psi_k Z_k^T (I-Q_kQ_k^T)\|^2
+ \|W_{\perp} \Psi_{\perp} Z_{\perp}^T (I-Q_kQ_k^T)\|^2 \\
& \le \|A\|^2 \cdot \sin^2(\calz_k, \calq_k)
+ \psi_{k+1}^2 \cdot \sin^2(\calz_{\perp}, \calq_k) \\
& \le \|A\|^2 \cdot \sin^2(\calz_k, \calq_k) + \psi_{k+1}^2.
\end{align*}
Furthermore, more specifically,
\[
\|A(I-Q_kQ_k^T)\|^2 = \sum_{j=1}^k \psi_j^2(A) \ |\bz_j^T (I-Q_kQ_k^T)|^2
+ \|W_{\perp} \Psi_{\perp} Z_{\perp}^T (I-Q_kQ_k^T)\|^2.\]
\end{proof}

The significance of this result is that $\|A(I-Q_kQ_k^T)\|$ may be close to $\psi_k(A)$ when $\calq_k$ captures the largest singular vectors of $A$ well. For instance, in the standard CUR, $Q_k$ is equivalent to $Z_k$ so the quantity $\sin^2(\calz_k, \calq_k)$ equals 0. If the matrix $B$ from \eqref{eq:mgsvd} is close to the identity or is a scaled identity, we expect that $\sin^2(\calz_k, \calq_k)$ will be approximately zero. However, this sine will generally not be small, as we illustrate by the following example.

\begin{example} {\rm
Let $A = \text{diag}(1,2,3)$, and $B = \text{diag}(1,20,300)$.
Denote by $\be_j$ the $j$th standard basis vector.
Then clearly $Z_1 = \bz_1 = \be_3$, while the largest right generalized singular vector $\bq_1$ is equal to the largest right singular vector of $AB^{-1} = \text{diag}(1, 0.1, 0.01)$, and hence $Q_1 = \bq_1 = \be_1$.
This implies that $\sin(\calz_1, \calq_1) = \sin(\bz_1, \bq_1)$ is large.}
\end{example}

\subsection{Error Bounds in terms of the GSVD approximation}
With the above results in mind, instead of using the rank-$k$ SVD approximation error, we will derive error bounds for $\norm{A-CMR}$ (see \eqref{eq:gcur}) in terms of the error bounds of a rank-$k$ GSVD approximation of $A$ (see \cref{pp1}).
The matrices $C$ and $R$ are of full-rank $k$ determined by the row and column index selection matrices $S$ and $P$, respectively and $M = C^+\!AR^+$.  From \cref{alg:GCUR-DEIM}, we know that $S$ and $P$ are derived using the $k$ columns of the matrices $U$ and $Y$, respectively, corresponding to the largest generalized singular value (see \eqref{eq:tgsvd}).

We use the interpolatory projector given in \cref{pp2}. Therefore instead of $Y$ (see \eqref{eq:mgsvd}), we use its orthonormal basis $Q$, to exploit the properties of an orthogonal matrix.

We will now analyze the approximation error between $A$ and its interpolatory projection $A \mathbb{P}$. The proof of the error bounds for the proposed method closely follows the one presented in \cite{Sorensen}.
The second inequality of the first statement of \cref{pp4} is in \cite[Lemma~4.1]{Sorensen}. The first inequality of the first statement is new but completely analogous. In the second statement, we use the GSVD.
For the analysis, we need the following QR-decomposition of $Y$ (see \eqref{eq:pgsvd}):
\begin{equation} \label{eq:QR}
[Y_k \ \ \wh Y] = Y = QT = [Q_k \ \ \wh Q]
\begin{bmatrix} T_k & T_{12} \\ 0 & T_{22} \end{bmatrix} = [Q_k T_k \ \ Q \wh T],
\end{equation}
where we have defined
\begin{equation} \label{eq:T_hat}
\wh T := \begin{bmatrix} T_{12} \\ T_{22} \end{bmatrix}.
\end{equation}
This implies that
\[
A = A_k + \wh U \, \wh \Gamma \, \wh Y^T = U_k \Gamma_k Y_k^T + \wh U \, \wh \Gamma \, \wh Y^T
= U_k \Gamma_k T_k^T Q_k^T + \wh U \, \wh \Gamma \, \wh T^T Q^T.
\]

\begin{proposition} \label{pp4} (Generalization of \cite[Lemma~4.1]{Sorensen})
Given $A \in \R^{m\times n}$ and $Q_k \in \R^{n\times k}$ with orthonormal columns where $k<n$, let $P \in \R^{n\times k}$ be a selection matrix and $Q_k^TP$ be nonsingular. Let $\mathbb{P} = P(Q_k^TP)^{-1}Q_k^T$, then
\[\psi_{\min}(A(I - Q_kQ_k^T))~\norm{(Q_k^TP)^{-1}}\le \norm{A-A\mathbb{P}} \leq \norm{A(I - Q_kQ_k^T)}~ \norm{(Q_k^TP)^{-1}}.\]
In particular, if $Q_k$ is an orthonormal basis for $Y_k$, the first $k$ columns of $Y$, then
\[
\gamma_{k+1}\cdot \psi_{\min}(T_{22}) \cdot \norm{(Q_k^TP)^{-1}}\le \norm{A-A\mathbb{P}} \le \gamma_{k+1} \cdot \norm{T_{22}} \cdot \norm{(Q_k^TP)^{-1}}.
\]
\end{proposition}
\begin{proof} We have that $Q_k^T\mathbb{P}=Q_k^TP(Q_k^TP)^{-1}Q_k^T=Q_k^T$ implies $Q_k^T(I-\mathbb{P}) = 0$.
Therefore,
\[\norm{A-A\mathbb{P}}=\norm{A(I-\mathbb{P})} = \norm{A(I - Q_kQ_k^T)(I-\mathbb{P})}\leq \norm{A(I - Q_kQ_k^T)}~\norm{I-\mathbb{P}}\]
and also
\[\norm{A(I - Q_kQ_k^T)(I-\mathbb{P})} \ge \psi_{\min}(A(I - Q_kQ_k^T))~\norm{I-\mathbb{P}} \ .\]
Note that, since $k<n$, we know that $\mathbb{P} \ne 0$ and $\mathbb{P} \ne I$, and hence (see, e.g., \cite{Daniel})
\[\norm{I-\mathbb{P}} =\norm{\mathbb{P}} =\norm{(Q_k^TP)^{-1}}.\]
With $A = U \, \Gamma \,Y^T$, $A_k = U_k\Gamma_kY_k^T$, $Y=QT$ and $Y_k=Q_kT_k$ we have
\begin{align*}
A\ Q_k\, Q_k^T &= \big[U_k \ \ \widehat U \big] \begin{bmatrix} \Gamma_k & 0\\ 0 & \widehat \Gamma \end{bmatrix} \begin{bmatrix} T_k^T & 0 \\[0.5mm] T_{12}^T & T_{22}^T \end{bmatrix} \begin{bmatrix} I_k \\ 0 \end{bmatrix} Q_k^T \\
&= U_k\Gamma_kT_k^TQ_k^T + \widehat U \, \widehat \Gamma \, T_{12}^TQ_k^T,
\end{align*}
and hence
\begin{align*}
A\,(I-Q_kQ_k^T) & = (A-A_k)-\widehat U \, \widehat \Gamma \, T_{12}^TQ_k^T \\
& = \wh U \, \wh \Gamma \, \wh T^T Q^T - \wh U \, \wh \Gamma \, T_{12}^TQ_k^T = \wh U \, \wh \Gamma \, T_{22}^T \wh Q^T.
\end{align*}
This implies
\[\norm{A\,(I - Q_kQ_k^T)} \le \gamma_{k+1} \cdot \norm{T_{22}}\]
and
\[\norm{A\,(I - Q_kQ_k^T)} \ge \gamma_{k+1} \cdot \psi_{\min}(T_{22}). \]
\end{proof}

Let us now consider the operation on the left-hand side of $A$. Given the set of interpolation indices $\{s_1, \dots, s_k\}$ determined from $U_k$, $S = [\be_{s_1},\dots,\be_{s_k}]$  and for a nonsingular $S^TU_k$, we have the DEIM interpolatory projector $\mathbb{S}=U_k(S^TU_k)^{-1}S^T$. Since $U_k$ consists of the dominant $k$ left generalized singular vectors of $A$ and has orthonormal columns, it is not necessary to perform a QR-decomposition as we did in \cref{pp4}.

The following proposition is analogous to \cref{pp4}. The results are similar to those in \cite[p.~A1461]{Sorensen} except that here, we use the approximation error of the GSVD instead of the SVD.
\begin{proposition}
\label{pp4b} Given $U_k \in \R^{m\times k}$ with orthonormal columns  where $k<m$, let $S \in \R^{m\times k}$ be a selection matrix and  $S^TU_k$ be nonsingular. Furthermore, let $\mathbb{S} = U_k(S^TU_k)^{-1}S^T$, then, with $\widehat T$ as in \eqref{eq:T_hat},
\[\gamma_{k+1}\cdot\psi_{\min}(\widehat{T})\cdot\norm{(S^TU_k)^{-1}}\le\norm{A-\mathbb{S}A}\leq \gamma_{k+1}\cdot\norm{\widehat{T}} \cdot \norm{(S^TU_k)^{-1}}. \]
\end{proposition}
\begin{proof}
We have
\[\norm{A-\mathbb{S}A} =\norm{(I-\mathbb{S})A}=\norm{(I-\mathbb{S})(I-U_kU_k^T)A}. \]
Similar to before, since $k<m$, we know that $\mathbb{S} \ne 0$ and $\mathbb{S} \ne I$ hence 
\[\norm{I-\mathbb{S}} =\norm{\mathbb{S}} =\norm{(S^TU_k)^{-1}}.\] 
Since $(I-U_k U_k^T) A = A-U_k \Gamma_k Y_k^T = \wh U \, \wh \Gamma \, \wh Y^T = \wh U \, \wh \Gamma \, \wh T^T\!Q^T$ we get
\[\norm{(I-U_kU_k^T)A}=\norm{A-A_k} \leq \gamma_{k+1}\cdot\norm{\wh{T}},\]
and $\norm{(I-U_kU_k^T)A} \ge \gamma_{k+1} \cdot \psi_{\min}(\wh{T})$, from which the result follows. 
\end{proof}

We will now use \cref{pp4,pp4b}  to find a bound for the approximation error of the GCUR of $A$ relative to $B$. As in \cite{Sorensen} we first show in the following proposition that the error bounds of the interpolatory projection of $A$ onto the chosen rows and columns apply equally to the orthogonal projections of $A$ onto the same row and column spaces. 
\begin{proposition}\label{pp5}(Generalization and slight adaptation of \cite[Lemma~4.2]{Sorensen}) Given the selection matrices $P$, $S$, let $C=AP$ and $R=S^T\!A$. Suppose that $C \in \R^{m \times k}$ and $R \in \R^ {k \times n}$ are full rank matrices with $k< \min(m,n)$, and that $Q_k^TP$ and $S^TU_k$ are nonsingular. With $\widehat T$ and $T_{22}$ as in \eqref{eq:QR}--\eqref{eq:T_hat}, we have the bound for the orthogonal projections of $A$ onto the column and row spaces: 
\begin{align*}
\norm{(I-CC^+)A} & \le \gamma_{k+1} \cdot \norm{T_{22}} \cdot \norm{(Q_k^TP)^{-1}}, \\
\norm{A(I-R^+\!R)} & \le \gamma_{k+1} \cdot \norm{\wh{T}} \cdot \norm{(S^TU_k)^{-1}}.
\end{align*}
\end{proposition}
\begin{proof} This proof is a minor modification of that of \cite[Lemma~4.2]{Sorensen}; we closely follow their proof technique.  With $C=AP$ of full rank, we have $C^+=(P^T\!A^T\!AP)^{-1}(AP)^T$. With this, the orthogonal projection of $A$ onto Range$(C)$ can be stated as
\[CC^{+\!}A=(AP(P^T\!A^T\!AP)^{-1}P^T\!A^T)A.\]
Let $\Pi_P =P(P^T\!A^T\!AP)^{-1}P^T\!A^T\!A$, note that $\Pi_P P = P$ since $\Pi_P$ is an oblique projector on Range$(P)$. We can rewrite $CC^+\!A$ as $CC^+\!A=A\Pi_P$.
Hence the error in the orthogonal projection of $A$  will be
$(I-CC^+)A=A(I-\Pi_P)$.
Since $\Pi_P  \mathbb{P}= \mathbb{P}$, we have
\[A(I-\Pi_P)=A(I-\Pi_P)(I-\mathbb{P})=(I-CC^+)A(I-\mathbb{P}),\]
therefore
\begin{align*}
\norm{(I-CC^+)A} &=\norm{A(I-\Pi_P)}=\norm{(I-CC^+)A(I-\mathbb{P})} \\
&\leq \norm{(I-CC^+)}~\norm{A(I-\mathbb{P})}.
\end{align*}
With $C$ being nonsquare, $\norm{I-CC^+} =1$ (see, e.g., \cite{Daniel})
and $\norm{A(I-\mathbb{P})}\leq \gamma_{k+1}\cdot \norm{T_{22}} \cdot \norm{(Q_k^TP)^{-1}}$ from \cref{pp4}, we have
\[\norm{(I-CC^+)A}\leq \gamma_{k+1}\cdot \norm{T_{22}} \cdot \norm{(Q_k^TP)^{-1}}.\]
In a similar vein, with $R=S^T\!A$ and $R^+=R^T(RR^T)^{-1}$ we have $R^+=A^TS(S^T\!AA^T\!S)^{-1}$ and the error in the orthogonal projection of $A$ is $A(I-R^+\!R) =(I-\Pi_S)A$, where $\Pi_S = AA^T\!S(S^T\!AA^T\!S)^{-1}S^T$, so that
\[(I-\Pi_S)A=(I-\mathbb{S})(I-\Pi_S)A=(I-\mathbb{S})A(I-R^+\!R)\]
and
\[
\norm{A(I-R^+\!R)}\leq \norm{(I-\mathbb{S})A}~\norm{(I-R^+\!R)} \leq \gamma_{k+1} \cdot \norm{\wh{T}} \cdot \norm{(S^TU_k)^{-1}}. 
\]
\end{proof}

This result helps to prove an error bound for the GCUR approximation error. For the following theorem, we again closely follow the approach of \cite{Sorensen} which also follows a procedure in \cite{Mahoney}.
As stated in \cref{Dfn4} the middle matrix can be computed as $M=(C^T\!C)^{-1}C^T\!A R^T(R R^T )^{-1}=C^+\!AR^+$.
\begin{theorem} \label{theorem1} (Generalization of \cite[Thm.~4.1]{Sorensen}) Given $A \in \R^{m\times n}$ and $Y_k$, $U_k$ from \eqref{eq:tgsvd}, let $P$ and $S$ be  selection matrices so that $C=AP$ and $R=S^T\!A$ are of full rank. Let $Q_k \in \R^{n\times k}$ be the $Q$-factor of $Y_k$, and $\widehat{T}$ and $T_{22}$ as in \eqref{eq:QR}--\eqref{eq:T_hat}. Assuming $Q_k^TP$ and $S^TU_k$ are nonsingular, then with the error constants
\[\eta_p := \norm{(Q_k^TP)^{-1}}, \qquad \eta_s := \norm{(S^TU_k)^{-1}},\]
we have
\begin{equation*}\label{eq:gcur_A}
\norm{A-CMR} \leq \gamma_{k+1} \cdot (\eta_p\cdot\norm{T_{22}} +\eta_s\cdot\norm{\wh{T}}) \le \gamma_{k+1} \cdot (\eta_p + \eta_s) \cdot \norm{\wh{T}}.  
\end{equation*}
\end{theorem}
\begin{proof}
By the definition of $M$, we have
\[A-CMR=A-CC^+\!AR^+\!R=(I-CC^+)A+CC^+\!A(I-R^+\!R),\]
using the triangle inequality, it follows that
\[\norm{A-CMR}=\norm{A-CC^+\!AR^+\!R}\leq \norm{(I-CC^+)A}+\norm{CC^+}~\norm{A(I-R^+\!R)}\]
and the fact that $CC^+$ is an orthogonal projection with $\norm{CC^+}=1$,
\[\norm{A-CMR} \le \gamma_{k+1} \cdot \norm{T_{22}} \cdot \norm{(Q_k^TP)^{-1}}+\norm{(S^TU_k)^{-1}}~\norm{\widehat{T}} \cdot \gamma_{k+1}.\]
\end{proof}

The last line of \cref{theorem1} can be related to the results in \cite[Thm.~4.1]{Sorensen}; both theorems have the factors $\eta_p$ and $\eta_s$. In \cite{Sorensen}, the error of the CUR approximation of $A$ is within a factor of $\eta_p + \eta_s$ of the best rank-$k$ approximation, obtained from the SVD. \Cref{theorem1} provides a bound in terms of $\gamma_{k+1}\le 1$ from the GSVD \eqref{gsvd} and the additional factors $\norm{\widehat{T}}$ and $\norm{T_{22}}$. The results presented in this section suggest that a good index selection procedure that yields small quantities $\norm{(Q_k^TP)^{-1}}$ and $\norm{(S^TU_k)^{-1}}$ is desirable. For a bound on $\norm{T_{22}}$, given $k$, Chandrasekaran and Ipsen \cite{chandrasekaran1994} have developed an efficient algorithm that computes a rank-revealing QR factorization
such that $\norm{T_{22}} \le \psi_{k+1} (Y)\sqrt{(k+1)(n-k)}$. To bound $\norm{\widehat{T}}$, we start by restating the results of \cite[Thm.~2.3]{hansen} for the GSVD of $(A,B)$. Defining
\[L := \begin{bmatrix} A \\ B \end{bmatrix}, \quad \text{it follows that} \quad \norm{X^{-1}}=\norm{L} \le \norm{A} + \norm{B}.\]
We know from \cref{eq:mgsvd} that $X^{-T}=Y$, so we can restate the above inequality as $\norm{Y}\le \norm{A} + \norm{B}$. Given the partitioning and QR factorization of $Y$ in \cref{eq:QR},
we have that
\[
\norm{\wh T}=\norm{Q\wh T}=\norm{\wh Y} \le \norm{Y} = \norm{L} \le \norm{A} + \norm{B}.
\]
In fact, we note that we can exploit the tighter bound $\|L\| \le (\|A\|^2+\|B\|^2)^{1/2}$ to improve the bound on $\wh T$ accordingly.

We note that where these results have been presented for matrix $A$ in \eqref{eq:gcur}, similar results can be obtained for $B$. The following error bound for the approximation of $B$ is analogous to \cref{eq:gcur_A}. As noted in \cref{Dfn4}, the selection matrix $P$ is similar for the GCUR decomposition of $A$ and $B$ therefore we have the quantity $\norm{(Q_k^TP)^{-1}}$ in the error bound of both factorizations:
\begin{align*}\label{eq:gcur_B}
\norm{B-C_BM_BR_B} &\leq \sigma_{k+1} \cdot ( \norm{(Q_k^TP)^{-1}}\cdot\norm{T_{22}} +\norm{(S_B^TV_k)^{-1}}\cdot\norm{\wh{T}}) \\
 &\le \sigma_{k+1} \cdot (\norm{(Q_k^TP)^{-1}} + \norm{(S_B^TV_k)^{-1}}) \cdot \norm{\wh{T}}.  
\end{align*}
It is worth nothing that these bounds hold irrespective of the approach used to select the row and column indices. Since the GCUR algorithm presented in this paper is DEIM-based, 
\cite{Sorensen} provides deterministic bounds:
\[\norm{(Q_k^TP)^{-1}}<\sqrt{\frac{nk}{3}}\,2^k, \quad \norm{(S_A^TU_k)^{-1}} <\sqrt{\frac{mk}{3}}\,2^k, \quad \text{and} \quad \norm{(S_B^TV_k)^{-1}} <\sqrt{\frac{dk}{3}}\,2^k.\]
We refer to \cite[Lemma~4.4]{Sorensen} for the constructive proofs, and will give an example with the various quantities in \cref{exp:1}.

\section{Numerical experiments}\label{sec:EXP}
We now present the results of a few numerical experiments to illustrate the performance of GCUR for low-rank matrix approximation. For the first two experiments, we consider a case where a data matrix $A$ is corrupted by a random additive noise $E$ and the covariance of this noise (the expectation of $E^T\!E$) is not a multiple of the identity matrix. We are therefore interested in a method that can take the actual noise into account. Traditionally, a pre-whitening matrix $R^{-1}$ (where $R$ is the Cholesky factor of the noise's covariance matrix) may be applied to the perturbed matrix \cite{hansen}, so that one can use SVD-based methods on the transformed matrix. With a GSVD formulation, the pre-whitening operation becomes an integral part of the algorithm \cite{hansen2007subspace}; we do not need to explicitly compute $R^{-1}$ and transform the perturbed matrix. We show in the experiments that using SVD-based methods without pre-whitening the perturbed data yields less accurate approximation results of the original matrix. 

For the last two experiments, we consider a setting with two data sets collected under different conditions, e.g., treatment and control experiment where the former has distinct variation caused by the treatment; signal-free and signal recordings with the signal-free data set containing only noise. We are interested in exploring and identifying patterns and discriminative features that are specific to one data set.  

\begin{experiment}\label{exp:1} {\rm
This experiment is an adaptation of experiments in \cite[p.~66: Sect.~3.4.4]{hansen} and \cite[Ex.~6.1]{Sorensen}; see also the motivating example in \cref{sec:intro}. We construct matrix $A$ to be of a known modest rank. We then perturb this matrix with a noise matrix $E \in \R^{m \times n}$ whose entries are correlated. Given $A_E= A+E$, we evaluate and compare the GCUR and the CUR decomposition on $A_E$  in terms of recovering the original matrix $A$. Specifically, the performance of each decomposition is assessed based on the 2-norm of the relative matrix approximation error i.e., $\norm{A-\widetilde {A} }/\norm{A}$, where $\widetilde {A}$ is the approximated low-rank matrix. 
We present the numerical results for four noise levels; thus $E=\varepsilon\, \frac{\norm{F}}{\norm{A}} F$ where $\varepsilon$ is the parameter for the noise level and $F$ is a randomly generated correlated noise. We first generate a sparse, nonnegative rank-50 matrix $A \in \R^{m \times n}$, with $m=100000$ and $n=300$, of the form 
\[A=\sum_{j=1}^{10}\frac{2}{j}\, \bx_j\ \by_j^T + \sum_{j=11}^{50}\frac{1}{j}\, \bx_j\ \by_j^T,\]
where $\bx_j \in \R^{m}$ and $\by_j \in \R^{n}$ are sparse vectors with random nonnegative entries (i.e., $\bx_j={\sf sprand}(m,1,0.025)$ and $\by_j={\sf sprand}(n,1,0.025)$, just as in \cite{Sorensen}. Unlike \cite{Sorensen} we then perturb the matrix with a correlated Gaussian noise $E$ whose entries have zero mean and a Toeplitz covariance structure (in MATLAB $\text{desired-cov}(F)={\sf toeplitz}(0.99^0$, $\dots$, $0.99^{n-1})$, $R={\sf chol}(\text{desired-cov}(F))$, and $F= {\sf randn}(m,n)\cdot R)$ and $\varepsilon \in \{0.05$, $0.1$, $0.15$, $0.2\}$. We compute the SVD of $A_E$  and the GSVD of $(A_E,R)$ to get the input matrices for the CUR and the GCUR decomposition respectively. \Cref{fig:a,fig:b,fig:c,fig:d} compare the relative errors of the proposed DEIM-GCUR (see \cref{alg:GCUR-DEIM}) and the DEIM-CUR (see \cref{alg:CUR-DEIM}) for reconstructing the low-rank matrix $A$ for different noise levels. 
We observe that for higher noise levels the GCUR technique gives a more accurate low-rank approximation of the original matrix $A$. The DEIM-GCUR scheme seems to perform distinctly well for higher noise levels and moderate values of $k$. As indicated in \cref{sec:GCUR}, the GCUR method is slightly more expensive since it requires the computation of the TGSVD instead of the TSVD. We observe that, as $k$ approaches rank$(A)$, the relative error of the TGSVD continues to decrease; this is not true for the GCUR. We may attribute this phenomenon to the fact that the relative error is saturated by the noise considering we pick actual columns and rows of the noisy data. Since $\varepsilon$ indicates the relative noise level, it is, therefore, natural that for increasing $k$, the quality of the TSVD approximation rapidly approaches $\varepsilon$. For this experiment, we assume that an estimate of the noise covariance matrix is known and therefore we have the exact Cholesky factor; we stress that this may not always be the case. 

Therefore, we now show an example where we use an inexact Cholesky factor $\wh R$. We derive $\wh R$ by multiplying all off-diagonal elements of the exact Cholesky factor $R$ by factors which are uniformly random from the interval $[0.9, 1.1]$. Here, the experiment setup is the same as described above with the difference that we compute the GSVD of $(A_E,\wh R)$ instead. In \cref{fig:2a,fig:2b}, we observe that the GCUR and the GSVD still deliver good approximation results even for an inexact Cholesky factor $\wh R$ which may imply that we do not necessarily need the exact noise covariance. 

\begin{figure}[ht!]
	\centering
	\subfloat[$\varepsilon=0.2$.]{\label{fig:a}{\includegraphics[width=0.48\textwidth]{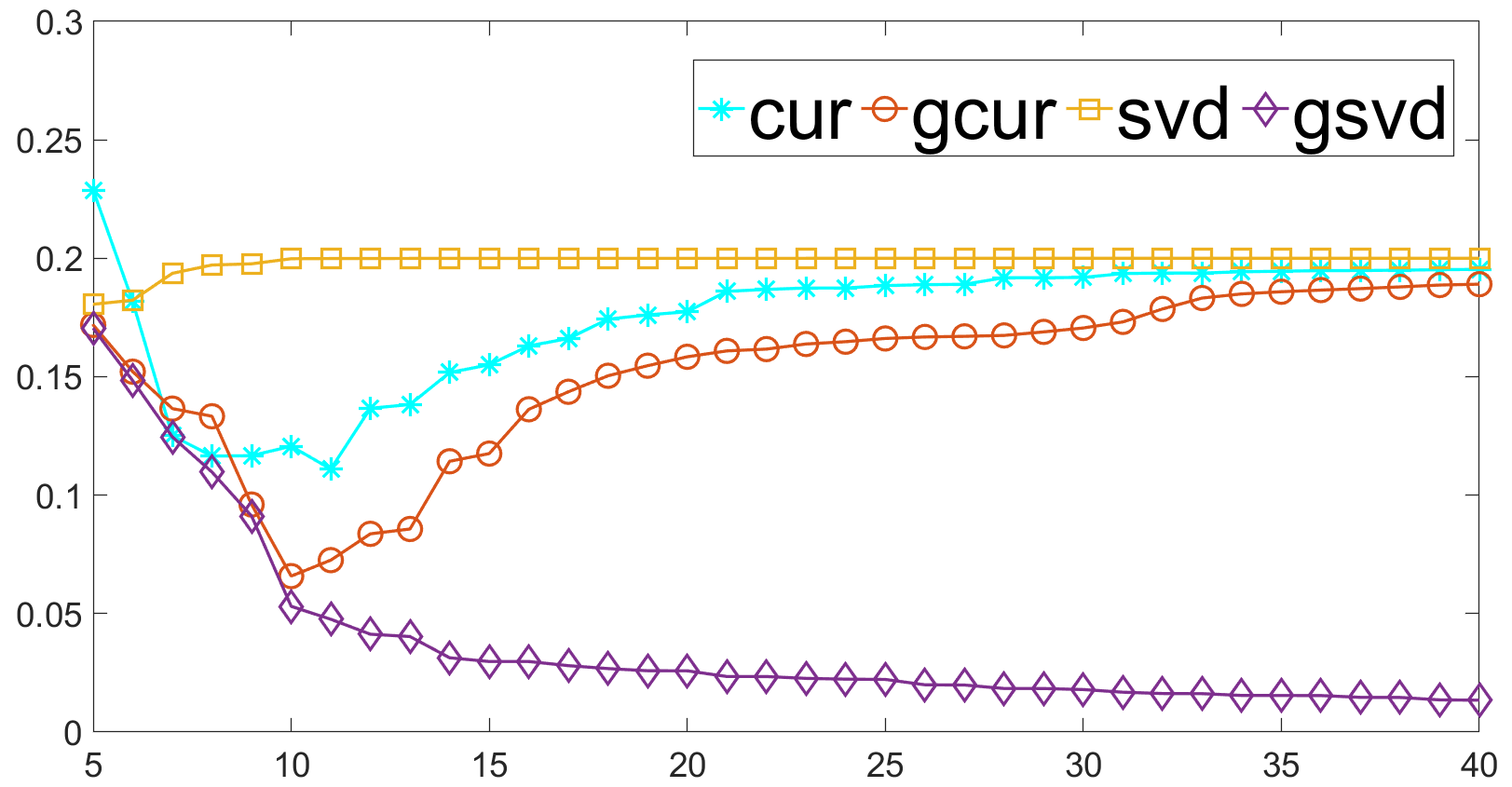}}}\hspace{0.1pt}
	\subfloat[$\varepsilon=0.15$.]{\label{fig:b}{\includegraphics[width=0.48\textwidth]{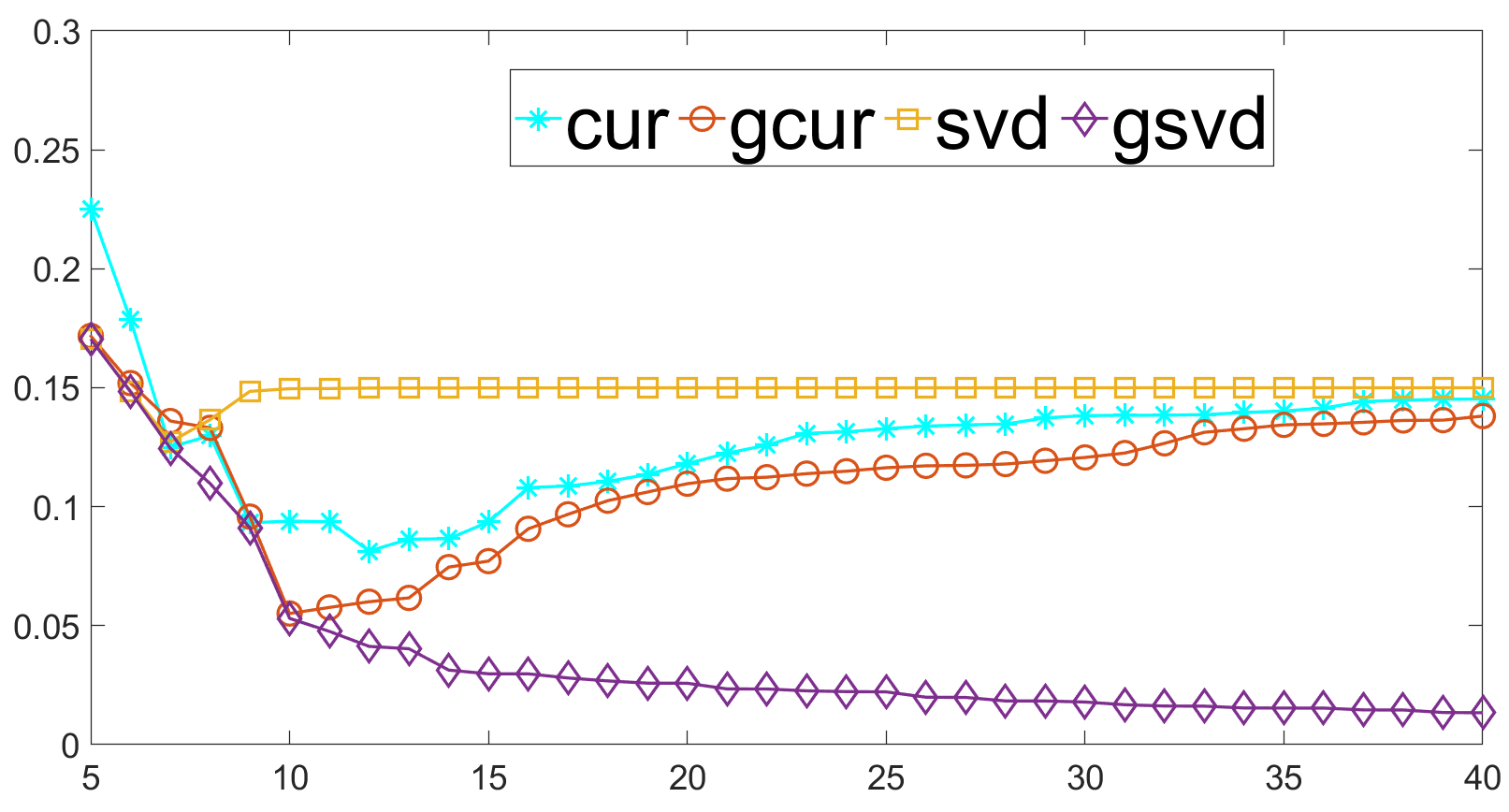} }}
	
	\subfloat[$\varepsilon=0.1$.]{\label{fig:c}{\includegraphics[width=0.48\textwidth,]{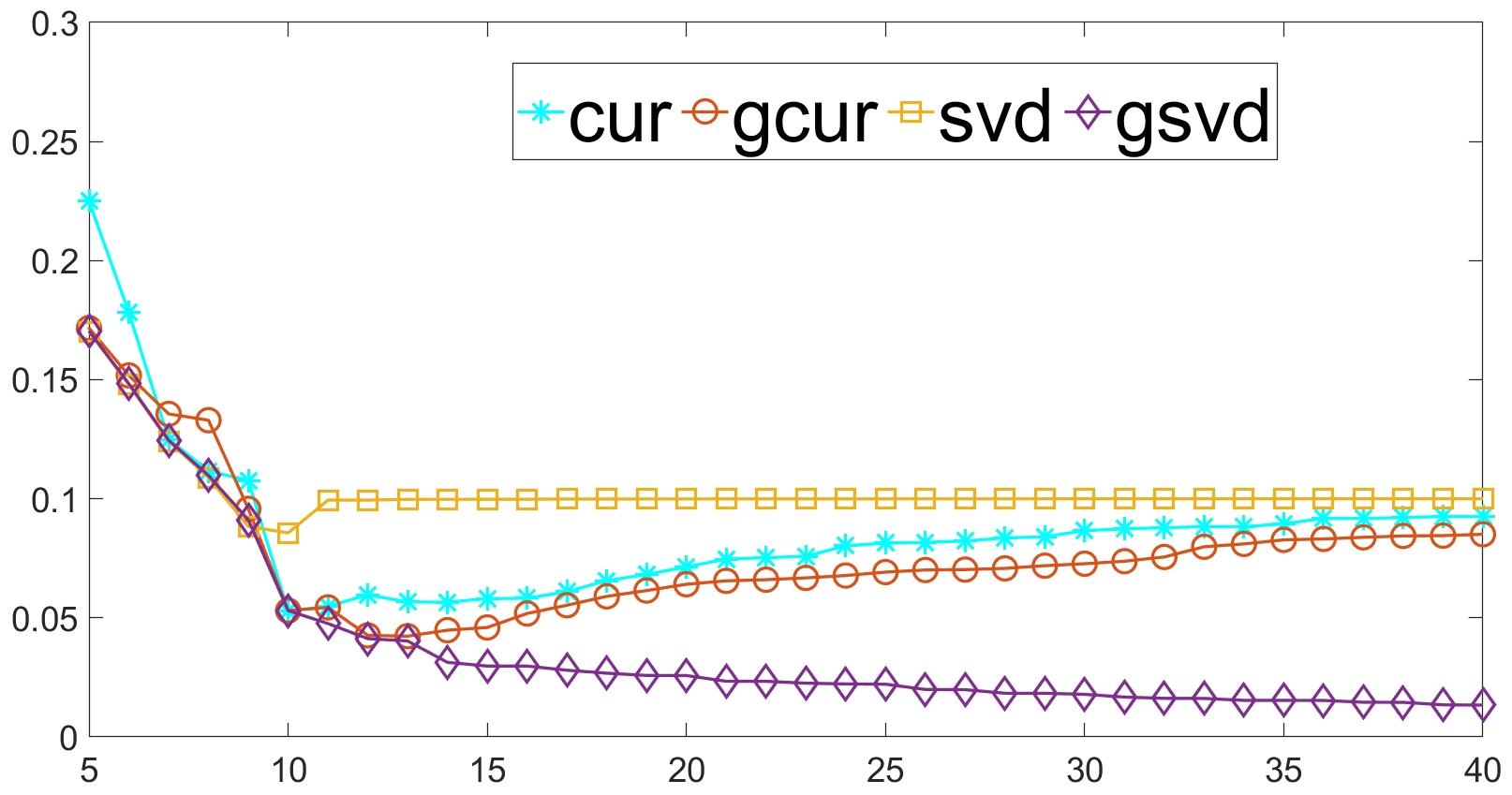} }}\hfill
	\subfloat[$\varepsilon=0.05$.]{\label{fig:d}{\includegraphics[width=0.48\textwidth]{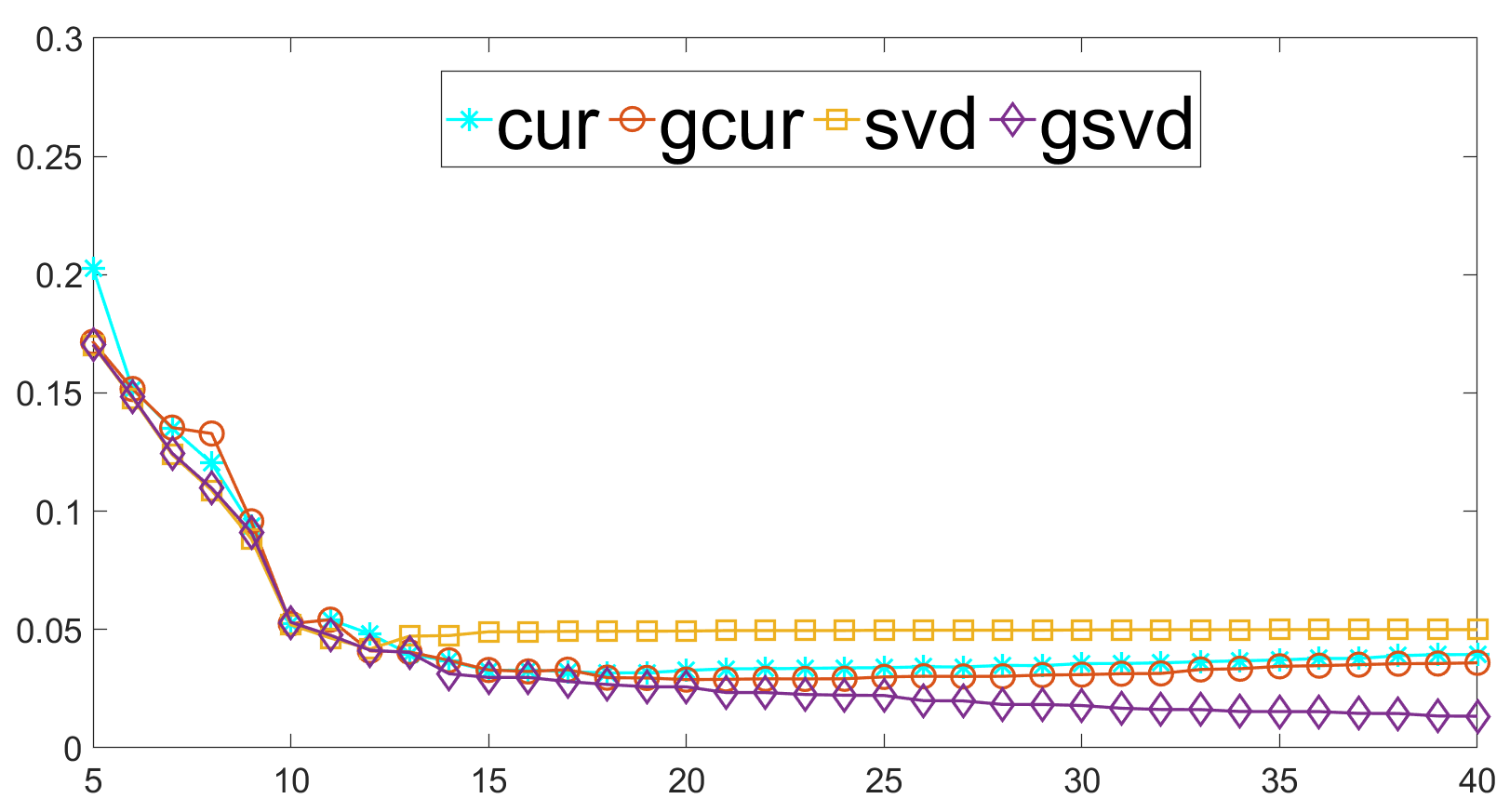} }}
	\caption{Accuracy of the DEIM-GCUR approximations compared with the standard DEIM-CUR approximations in recovering a sparse, nonnegative matrix $A$ perturbed with correlated Gaussian noise (\cref{exp:1}) using exact Cholesky factor of the noise covariance. The relative errors $\norm{A-\widetilde {A}_k }/\norm{A}$ (on the vertical axis) as a function of rank $k$ (on the horizontal axis) for $\varepsilon=0.2$, $0.15$, $0.1$, $0.05$, respectively.\label{fig:1}}
\end{figure}

\begin{figure}[ht!]
	\centering
	\subfloat[$\varepsilon=0.15$.]{\label{fig:2a}{\includegraphics[width=0.48\textwidth]{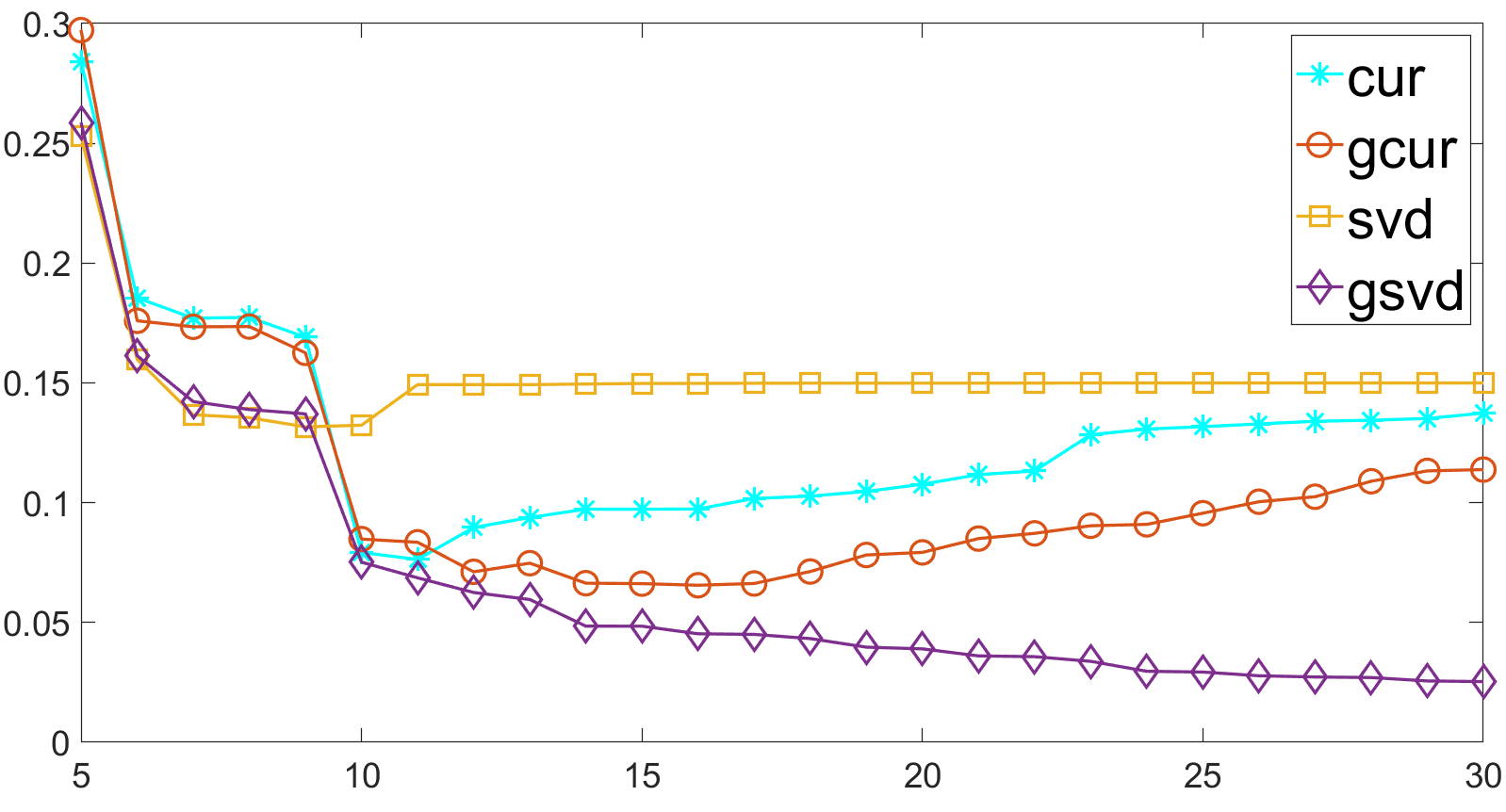}}}\hspace{0.1pt}
	\subfloat[$\varepsilon=0.1$.]{\label{fig:2b}{\includegraphics[width=0.48\textwidth]{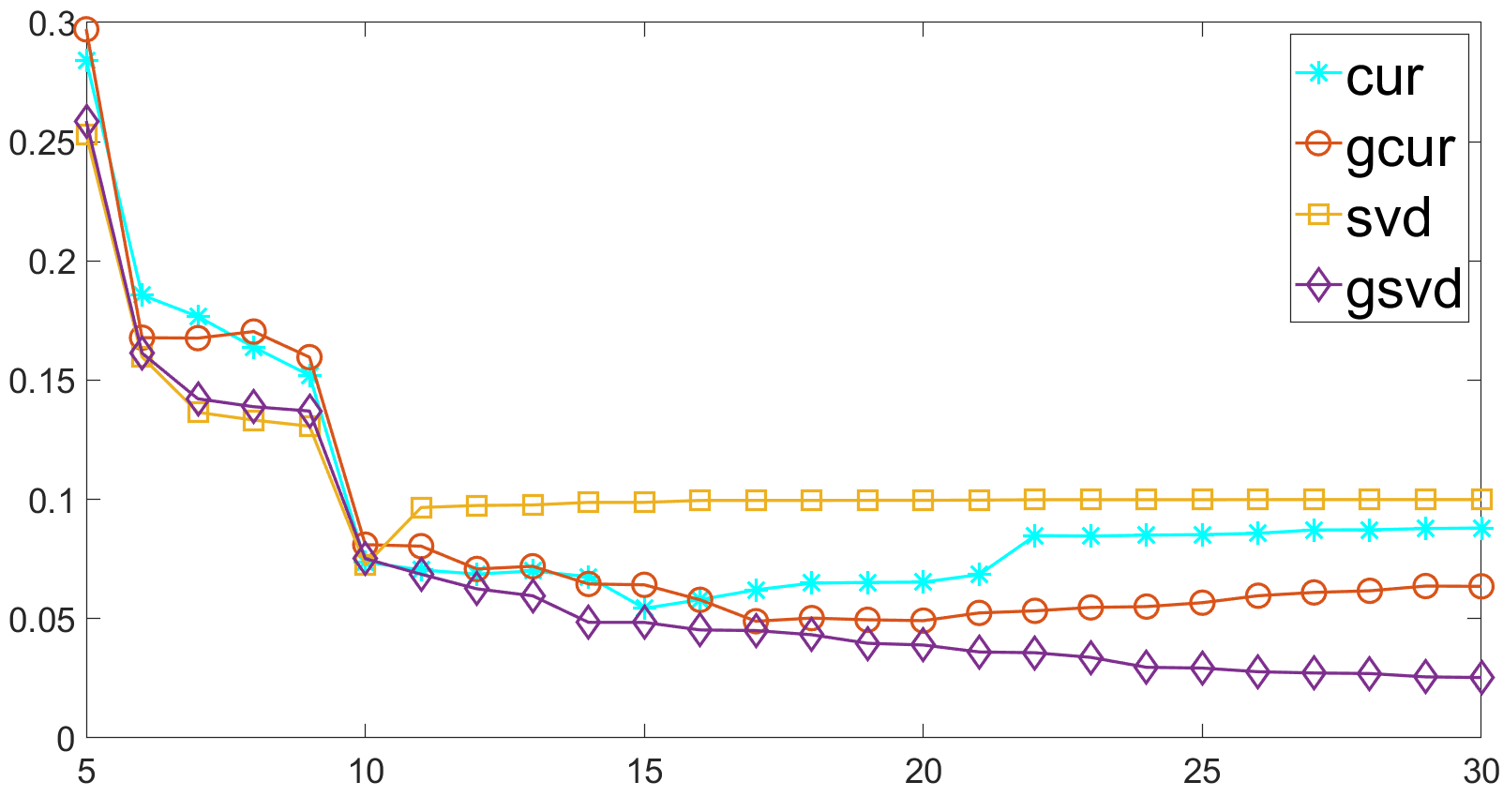} }}
	
	\caption{Accuracy of the DEIM-GCUR approximations compared with the standard DEIM-CUR approximations in recovering a sparse, nonnegative matrix $A$ perturbed with correlated Gaussian noise (\cref{exp:1}) using an inexact Cholesky factor of the noise covariance. The relative errors $\norm{A-\widetilde {A}_k }/\norm{A}$ (on the vertical axis) as a function of rank $k$ (on the horizontal axis) for $\varepsilon=0.15$, $0.1$, respectively.\label{fig:2}}
\end{figure}

We conclude this experiment by an illustration of the various quantities in \cref{theorem1}.
In \cref{fig:6}, we see that the upper bound in \cref{theorem1} may be a rather crude bound on the true GCUR error. As in \cite[Fig.~4]{Sorensen}, the quantities $\eta_S$ and $\eta_P$ may differ considerably in magnitude. While $\|T_{22}\|$ steadily decreases, $\|\wh T\|$ seems to stabilize as $k$ increases.

\begin{figure}[ht!]
	\centering
\includegraphics[width=0.5\textwidth, height=5cm]{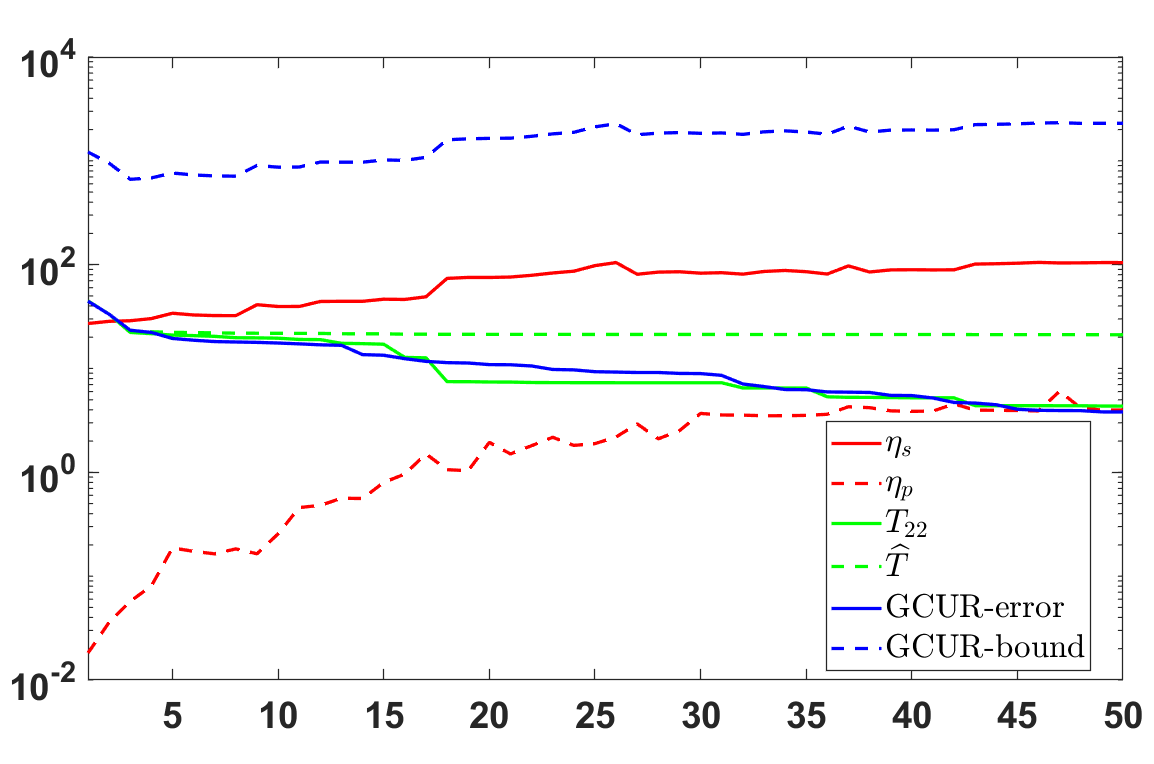}
\caption{Various quantities from \cref{theorem1}: error constants $\eta_P = \norm{(Q_k^TP)^{-1}}$ (red dashed) and $\eta_S = \norm{(S_A^TU_k)^{-1}}$ (red solid); multiplicative factors $\norm{T_{22}}$ (green solid) and $\norm{\wh T}$ (green dashed); the GCUR true error $\|A_E-(CMR)_{\sf gcur}\|$ of approximating $A_E$ in \cref{exp:1} (blue solid) and its upper bound (blue dashed).\label{fig:6}}
\end{figure}

}
\end{experiment}

\begin{experiment}\label{exp:2}{\rm For this experiment, we maintain all properties of matrix $A_E$ mentioned in the preceding experiment except for the column size that we reduce to 10000 (i.e., $A_E \in \R^{10000 \times 300}$) and instead of a sparse nonnegative matrix $A$, we generate a dense random matrix $A$. As in \cite[Ex.~6.2]{Sorensen}, we also modify $A$ so that there is a significant drop in the 10th and 11th singular values. The matrix $A$ is now of the form  
\[A=\sum_{j=1}^{10}\frac{1000}{j}\, \bx_j\ \by_j^T + \sum_{j=11}^{50}\frac{1}{j}\, \bx_j\ \by_j^T,\]
For each fixed $\varepsilon$ and $k$, we repeat the process 100 times and then compute the average relative error. The results in \cref{tab:1} show that the advantage of the GCUR over the CUR still remains even when singular values of the original matrix $A$ decrease more sharply. We observe that the difference in the relative error of the GCUR and the CUR is quite significant when the rank of the recovered matrix $\widetilde A$ is lower than that of $A$ (i.e., $k \ll 50$).

\begin{table}[ht!]

\centering
{\caption{Comparison of the qualities $\norm{A-\widetilde {A}_k }/\norm{A}$ of the TSVD, TGSVD, CUR, and GCUR approximations as a function of index $k$ and noise level $\varepsilon$ in \cref{exp:2}. The relative errors are the averages of 100 test cases.}\label{tab:1}
{\footnotesize
\begin{tabular}{clcccc} \hline \rule{0pt}{3mm}
$k$ & Method $\backslash \ \varepsilon$ & $0.05$ & $0.1$ & $0.15$ & $0.2$ \\ \hline \rule{0pt}{3.5mm}
10 & TSVD & $0.008$ &   $0.045$    & $0.150$   & $0.200$\\
& TGSVD & $0.002$ &    $0.003$ &    $0.005$ & $0.007$\\[0.5mm]
& CUR & $0.052$   & $0.118$ & $0.141$ &  $0.186$\\
& GCUR & $0.053$ &   $0.088$ &   $0.112$ &   $0.134$\\[0.5mm] \hline \rule{0pt}{3.5mm}
15 & TSVD & $0.050$ &   $0.100$    & $0.150$   & $0.200$\\
& TGSVD & $0.009$ &    $0.017$ &    $0.026$ & $0.035$\\[0.5mm]
& CUR & $0.049$   & $0.097$ & $0.146$ &  $0.196$\\
& GCUR & $0.046$ &   $0.091$ &   $0.138$ &   $0.185$\\[0.5mm] \hline \rule{0pt}{3.5mm}
20 & TSVD & $0.050$ &  $0.100$  &  $0.150$   &  $0.200$\\
& TGSVD & $0.011$  &  $0.023$  &  $0.034$  &  $0.015$\\[0.5mm]
& CUR & $0.050$  &  $0.099$  &  $0.149$  &  $0.199$\\
& GCUR & $0.049$   &  $0.097$  &  $0.146$  & $ 0.198$ \\[0.5mm] \hline
\rule{0pt}{3.5mm}
30 & TSVD & $0.050$ &   $0.100$    & $0.150$  &  $0.200$\\
& TGSVD & $0.016$ &    $0.031$ &    $0.047$ & $0.063$\\[0.5mm]
& CUR & $0.050$   & $0.100$ & $0.150$ &  $ 0.199$\\
& GCUR & $0.050$ &   $0.099$ &   $0.149$ &   $0.199$\\ \hline
\end{tabular}}}
\end{table}

The higher the noise level, the more advantageous the GCUR scheme may be over the CUR one. Especially for moderate values of $k$ such as $k=10$, the GCUR approximations are of better quality than those based on the CUR. For higher values of $k$ such as $k=30$, the approximation quality of the CUR and GCUR method become comparable since they both start to pick up the noise in the data columns. In this case, the GCUR does not improve on the CUR. Since it is a discrete method, picking indices for columns instead of generalized singular vectors, we see that the GCUR method yields worse results than the TGSVD approach.
}
\end{experiment}

\begin{experiment}\label{exp:3}{\rm Our next experiment is adapted from \cite{Abid}. We create synthetic data sets which give an intuition for settings where the GSVD and the GCUR may resolve the problem of subgroups. Consider a data set of interest (target data), $A$, containing 400 data points in a 30-dimensional feature space. This data set has four subgroups ({\color{aqua} blue}, {\color{yellow} yellow}, {\color{orange} orange}, and {\color{darkmagenta} purple}), each of 100 data points. The first 10 columns for all 400 data points are randomly sampled from a normal distribution with a mean of 0 and a variance of 100. The next 10 columns of two of the subgroups ({\color{aqua} blue} and {\color{orange} orange}) are randomly sampled from a normal distribution with a mean of 0 and a unit variance  while the other two subgroups ({\color{yellow} yellow} and {\color{darkmagenta} purple}) are randomly sampled from a normal distribution with a mean of 6 and a unit variance. The last 10 columns of subgroups {\color{aqua} blue} and {\color{yellow} yellow} are sampled from a normal distribution with a mean of 0 and a unit variance and those of {\color{darkmagenta} purple} and {\color{orange} orange} are sampled from  a normal distribution with a mean of 3 and a unit variance.

One of the goals of the SVD (or the related concept principal component analysis) in dimension reduction is to find a low-dimensional rotated approximation of a data matrix while maximizing the variances.
We are interested in reducing the dimension of $A$. If we project the data onto the two leading right singular vectors, we are unable to identify the subgroups because the variation along the first 10 columns is significantly larger than in any other direction, so some combinations of those columns are selected by the SVD. 

Suppose we have another data set $B$ (a background data set), whose first 10 columns are sampled from a normal distribution with a mean of 0 and a variance of 100, the next 10 columns are sampled from a normal distribution with a mean of 0 and a variance of 9 and the last 10 columns are sampled from a normal distribution with a mean of 0 and a unit variance. The choice of the background data set is key in this context. Generally, the background data set should have the structure we would like to suppress in the target data, which usually corresponds to the direction with high variance but not of interest for the data analysis \cite{Abid}. With the new data, one way to extract discriminative features for clustering the subgroups in $A$ is to maximize the variance of $A$ while minimizing that of $B$, which leads to a trace ratio maximization problem \cite{chen}
\[\widehat U := \argmax_{U \in \R^{n \times k}, \ U^TU=I_k}~ \text{Tr}~\big[(U^T\!B^T\!B\,U)^{-1}(U^T\!A^T\!A\,U)\big],\]
where $n=30$ and $k=5$ or $k=10$.  
By doing this, the first dimensions are less likely to be selected because they also have a high variance in data set $B$. Instead, the middle and last dimensions of $A$ are likely to be selected as they have the dimensions with the lowest variance in $B$, thereby allowing us to separate all four subgroups. The solution $\widehat U \in \R^{n \times k}$ to the above problem is given by the $k$ (right) eigenvectors of $(B^T\!B)^{-1}(A^T\!A)$ corresponding to the $k$ largest eigenvalues (cf., \cite[pp.~448--449]{fukunaga2013}); this corresponds to the (``largest'') right generalized singular vectors of $(A,B)$ (the transpose of \eqref{eqn:matrix_Y}). As seen in \cref{fig:3}, projecting $A$ onto the leading two right generalized singular vectors produces a much clearer subgroup separation (top-right figure) than projecting onto the leading two right singular vectors (top-left figure). Therefore, we can expect that a CUR decomposition based on the SVD will also perform not very well with the subgroup separation. In the bottom figures is a visualization of the data using the first two important columns selected using the DEIM-CUR (left figure) and the DEIM-GCUR (right figure). To a large extent, the GCUR is able to differentiate the subgroups while the CUR fails to do so. We investigate this further by comparing the performance of subset selection via DEIM-CUR on $A$ (\cref{alg:CUR-DEIM}) and DEIM-GCUR on $(A, B)$ (\cref{alg:GCUR-DEIM}) in identifying the subgroup or class representatives of $A$; we select a subset of the columns of $A$ (5 and 10) and compare the classification results of each method. We center the data sets by subtracting the mean of each column from all the entries in that column. Given the class labels of the subgroups, we perform a ten-fold cross-validation (i.e., split the data points into 10 groups and for each unique group take the group as test data and the rest as training \cite[p.~181]{james2013introduction}) and apply two classifiers on the reduced data set: ECOC (Error Correcting Output Coding) \cite{dietterich1994} and {\em classification tree} \cite{banfield2006} using the functions {\sf fitcecoc} and {\sf fitctree} with default parameters as implemented in MATLAB. It is evident from \cref{tab:2} that the TGSVD and the GCUR achieve the least classification error rate, e.g., for reducing the dimension from 30 to 10; 0\% and 6.3\% respectively, using the ECOC classifier and 0\% and 9.5\%  respectively, using the tree classifier. The standard DEIM-CUR method achieves the worst classification error rate.

\begin{figure}[ht!]
    \centering
    \includegraphics[width=\textwidth]{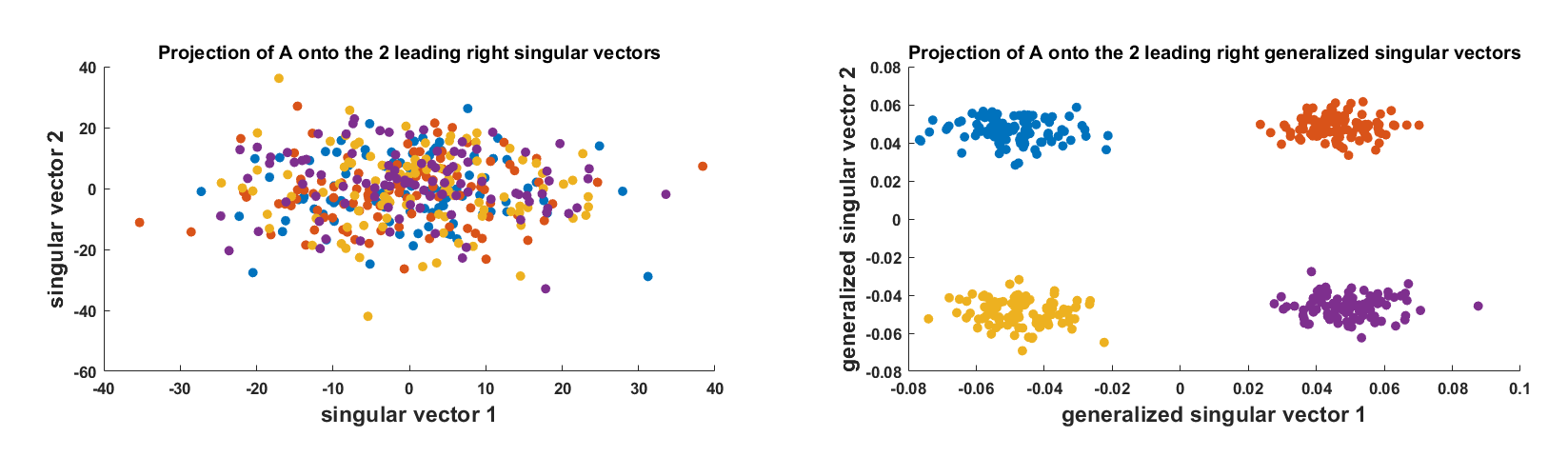}\\
    \includegraphics[width=\textwidth]{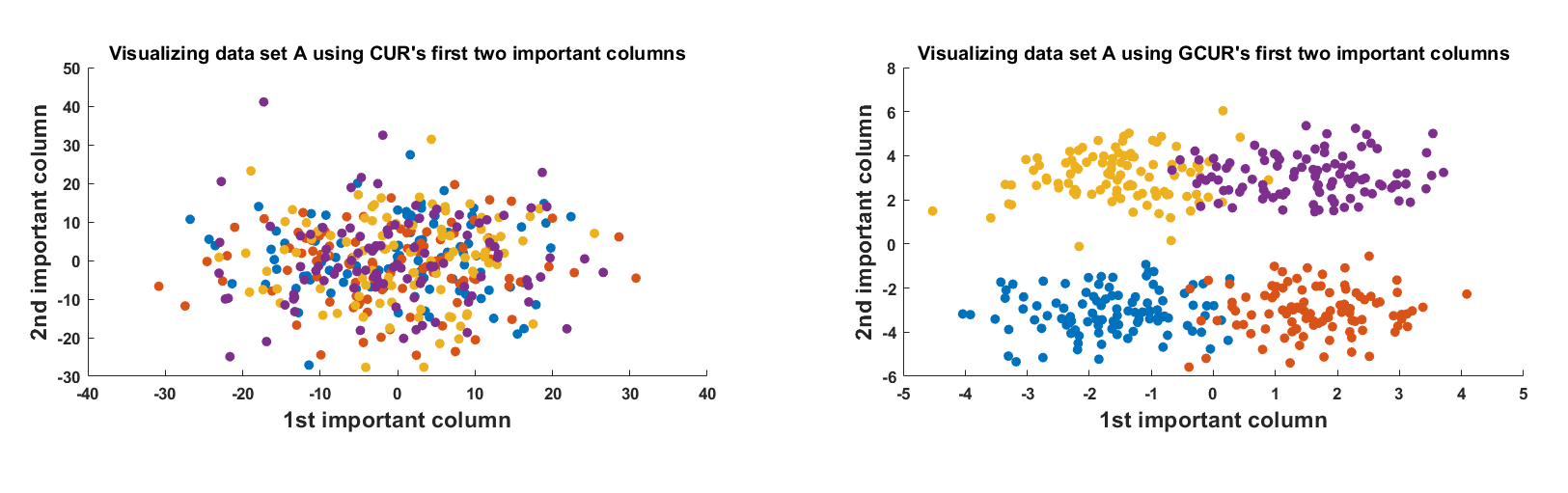}
    \caption{(Top-left) We project the synthetic data containing four subgroups onto the first two dominant right singular vectors. The lower-dimensional representation using the SVD does not effectively separate the subgroups. In the bottom-left figure, we visualize the data using the first two columns selected by DEIM-CUR. (Top-right) We illustrate the advantage of using GSVD by projecting the data onto the first two dominant right generalized singular vectors corresponding to the two largest generalized singular values. In the bottom-right figure, we visualize the data using the first two columns selected by DEIM-GCUR. The lower-dimensional representation of the data using the GSVD-based methods clearly separates the four clusters while the SVD-based methods fail to do so.}
    \label{fig:3}
\end{figure}

\begin{table}[ht!]
\centering
{\caption{$k$-Fold loss is the average classification loss overall 10-folds using SVD, GSVD, CUR, and GCUR as dimension reduction in \cref{exp:3}. The second and third columns give information on the number of columns selected from the data set using the CUR and GCUR plus the number of singular and generalized singular vectors considered for the ECOC classifier. Likewise, the fifth and sixth columns for the tree classifier.}\label{tab:2}
{\footnotesize
\begin{tabular}{lcclcc} \hline \rule{0pt}{3mm}%
Method & \multicolumn{2}{c}{$k$-Fold Loss} & Method & \multicolumn{2}{c}{$k$-Fold Loss}\\
 & 5 & 10 & & 5 &10\\ \hline \rule{0pt}{3.5mm}%
TSVD+ECOC &$0.638$ &$0.490$ & TSVD+Tree  & $0.693$ & $0.555$  \\
TGSVD+ECOC & $0\phantom{.000}$& $0\phantom{.000}$ & TGSVD+Tree  & $0\phantom{.000}$ & $0\phantom{.000}$ \\[0.5mm]
CUR+ECOC & $0.793$ &$0.485$  & CUR+Tree  & $0.793$  & $0.540$  \\
 GCUR+ECOC & $0.055$&$0.063$  & GCUR+Tree & $0.075$  &$0.095$   \\[0.5mm] \hline
\end{tabular}}}

\end{table}
}
\end{experiment}
\begin{experiment}\label{exp:4}{\rm We will now investigate the performance of the GCUR compared to the CUR on a higher-dimensional public data sets. The data sets consists of single-cell RNA expression levels of bone marrow mononuclear cells (BMMCs) from an acute myeloid leukemia (AML) patient and two healthy individuals. We have data on the BMMCs before stem-cell transplant and the BMMCs after stem-cell transplant. We preprocess the data sets as described by the authors in \cite{boileau2020}\footnote{\url{https://github.com/PhilBoileau/EHDBDscPCA/blob/master/analyses/}} keeping the 1000 most variable genes measured across all 16856 cells (patient-035: 4501 cells and two healthy individuals; one of 1985 cells and the other of 2472 cells). The data from the two healthy patients are combined to create a background data matrix of dimension $4457 \times 1000$ and we use patient-035 data set as the target data matrix of dimension $4501 \times 1000$ . Both data matrices are sparse: the  patient-035 data matrix has 1,628,174 nonzeros; i.e., about 36\% of all entries are nonzero and the background data matrix has 1,496,229 nonzeros; i.e., about 34\% of all entries are nonzero. We are interested in exploring the differences in the AML patient's BMMC cells  pre- and post-transplant.
We perform SVD, GSVD, CUR and GCUR on the target data (AML patient-035) to see if we can capture the biologically meaningful information relating to the treatment status. For the GSVD and the GCUR procedure the background data is taken into account. As evident in \cref{fig:4}, the GSVD and the GCUR produce almost linearly separable clusters which corresponds to pre- and post-treatment cells. These methods evidently capture the biologically meaningful information relating to the treatment and are more effective at separating the pre- and post-transplant cell samples compared to the other two. For the SVD and the CUR, we observe that both cell types follow a similar distribution in the space spanned by the first three dominant right singular vectors and the first three important gene columns, respectively. Both methods fail to separate the pre- and post-transplant cells.  
\begin{figure}[ht!]
    \centering
    \includegraphics[width=\textwidth]{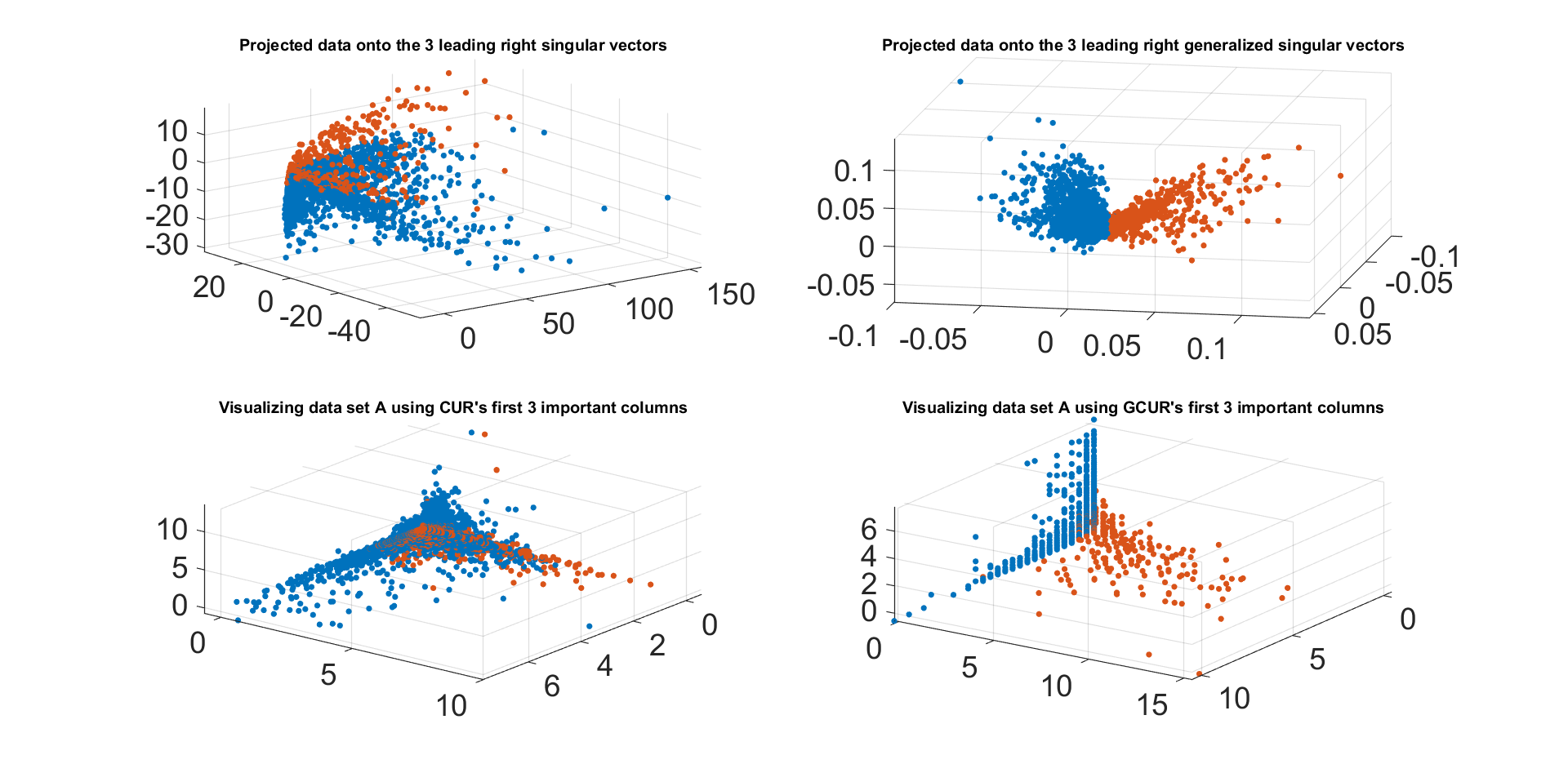}
    \caption{Acute myeloid leukemia patient-035 scRNA-seq data.(Top-left) A 3-D projection of the patient's BMMCs on the first three dominant right singular vectors. In the bottom-left figure, we visualize the data using the first three genes selected by DEIM-CUR. The lower-dimensional representation using the SVD-based methods does not effectively give a discernible cluster of the pre- and post-transplant cells. (Top-right) We illustrate the advantage of using GSVD by projecting the patient's BMMCs onto the first three dominant right generalized singular vectors corresponding to the three largest generalized singular values. In the bottom-right figure, we visualize the data using the first three genes selected by DEIM-GCUR. The lower-dimensional representation using the GSVD-based methods produce a linearly separable clusters.}
    \label{fig:4}
\end{figure}
}
\end{experiment}
\section{Conclusions}\label{sec:Con} In this paper we propose a new method, the DEIM-induced GCUR (generalized CUR) factorization with pseudocode in \cref{alg:GCUR-DEIM}. It is an extension of the DEIM-CUR decomposition for matrix pairs.
Just as the CUR decomposition has an interpolative decomposition (see, e.g., \cite{Voronin}) associated with it, there is a {\em generalized interpolative decomposition} (see \eqref{eq:ID})
associated with the GCUR decomposition.

When $B$ is square and nonsingular, there are close connections between the GCUR of $(A,B)$ and the DEIM-induced CUR of $AB^{-1}$. When $B$ is the identity, the GCUR decomposition of $A$ coincides with the DEIM-induced CUR decomposition of $A$. There exist a similar connection between the CUR of $AB^+$ and the GCUR of $(A, B)$ for a nonsquare but full-rank matrix $B$. 

While a CUR decomposition acts on one data set, a GCUR factorization decomposes two data sets together. An implication of this is that we can use it in selecting discriminative features of one data set relative to another. For subgroup discovery and subset selection in a classification problem where two data sets are available, the new method can perform better than the standard DEIM-CUR decomposition as shown in the numerical experiments. The GCUR algorithm can also be useful in applications where a data matrix suffers from non-white (colored) noise. The GCUR algorithm can provide more accurate approximation results compared to the DEIM-CUR algorithm when recovering an original matrix with low rank from data with colored noise. For the recovery of data perturbed with colored noise, we need the Cholesky factor of an estimate of the noise covariance. However, as shown in the experiments, even for an inexact Cholesky factor the GCUR may still give good approximation results. We note that, while the GSVD always provides a more accurate result than the SVD regardless of the noise level, the GCUR decomposition is particularly attractive for higher noise levels and moderate values of the rank of the recovered matrix compared to the CUR factorization. In other situations, both methods may provide comparable results. In addition, the GCUR decomposition is a discrete method, so choosing indices for columns and rows instead of the generalized singular vectors leads to worse results than the GSVD approach. 

Although we used the generalized singular vectors here, in principle one could use other vectors, e.g., an approximation to the generalized singular vectors. In our experiments, we choose the same number of columns and rows for the approximation of the original matrix. However, we do not need to choose the same number of columns and rows. We have extended the existing theory concerning the DEIM-CUR approximation error to this DEIM generalized CUR factorization; we derived the bounds of a rank-$k$ GCUR approximation of $A$ in terms of a rank-$k$ GSVD approximation of $A$.

Instead of the DEIM procedure for the index selection from the GSVD, it might be possible to use alternative index selection strategies. For a CUR decomposition, one alternative approach to DEIM is to perform a QR factorization with column pivoting \cite{Drmac} on the transpose of the matrices from the truncated SVD. In \cite{Goreinov2010,goreinov1997}, the authors propose a CUR factorization where $C$ and $R$ are selected by searching for submatrices of maximal volume in the singular vector matrices.
Another popular approach is selection via leverage score sampling \cite{Drineas,Mahoney}.
It may be interesting to extend these strategies to the context of the GCUR.


Computationally, the DEIM-GCUR algorithm requires the input of the GSVD, which is of the same complexity but more expensive than the SVD required for DEIM-CUR. For the case where we are only interested in approximating the matrix $A$ from the pair $(A, B)$, we can omit some of the lines in \cref{alg:GCUR-DEIM}; thus saving computational cost. 
In the case that the matrices $A$ and $B$ are so large, a full GSVD may not be affordable; in this case, we can consider iterative methods (see, e.g., \cite{Zwaan, hochstenbach, Zha}).

While in this work we used the GCUR method in applications such as extracting information from one data set relative to another, we expect that its promise may be more general.

\medskip\noindent
{\bf Acknowledgment}: the authors thank the referees for their very positive and helpful expert suggestions, which significantly improved the paper.

\bibliographystyle{siam}
\bibliography{references/gcur}

\end{document}